\documentclass[a4paper,fleqn]{cas-sc}

\usepackage[numbers]{natbib}

\usepackage{bm}
\usepackage{amssymb}
\usepackage{amsfonts}
\usepackage{stackrel}
\usepackage{setspace}
\usepackage{stackengine}
\usepackage{graphicx,placeins,float}
\usepackage{amsthm}
\usepackage{tabularx}
\usepackage{epstopdf}
\usepackage{subcaption}

\usepackage{lineno}

\numberwithin{equation}{section}

\newcommand{\norm}[1]{\left\lVert#1\right\rVert}

\newcommand{\grad}[1]{\textbf{grad}\,{#1}}

\newtheorem{remark}{Remark}

\newtheorem{proposition}{Proposition}

\begin{document}
\let\WriteBookmarks\relax
\def\floatpagepagefraction{1}
\def\textpagefraction{.001}
\shorttitle{Energy stable CGFEM framework on SBP form for INS equations}
\shortauthors{Mandal M. et~al.}

\title [mode = title]{A high order accurate and energy stable continuous Galerkin framework on summation-by-parts form for the incompressible Navier-Stokes equations}

\author[1]{Mrityunjoy Mandal}

\affiliation[1]{organization={InCFD Research Group, Department of Mechanical Engineering, University of Cape Town}}

\author[1]{Arnaud {G Malan}}
\cormark[1]
\ead{arnaud.malan@uct.ac.za}
\cortext[cor1]{Corresponding author}

\author[2]{Prince Nchupang}
\affiliation[2]{organization={Department of Mathematical Sciences, Applied Mathematics Division, Stellenbosch University}}

\author[3,4]{Jan Nordstr\"{o}m}
\affiliation[3]{organization={Department of Mathematics, Applied Mathematics, Link\"{o}ping University}}
\affiliation[4]{organization={Department of Mathematics and Applied Mathematics}, University of Johannesburg}



	%


\begin{abstract}
This paper presents a high-order accurate Continuous Galerkin Finite Element Method (CGFEM) for solving the initial boundary value problems governed by the Incompressible Navier–Stokes (INS) equations. We discretize the INS equations using the CGFEM approach in Summation-By-Parts (SBP) form. Lagrange polynomials of up to $4^{\textrm{th}}$ order are employed. The boundary conditions are imposed weakly using the Simultaneous Approximation Term (SAT) technique, which accommodates discontinuous boundary data without special treatment. The resulting SBP-SAT formulation guarantees an energy stable discretization. The efficiency of the proposed framework is demonstrated by solving a series of numerical tests. Initially, the Method of Manufactured Solutions (MMS) is employed to demonstrate $4^{\textrm{th}}$ order convergence. Subsequently, the $4^{\textrm{th}}$ order accurate scheme is applied to a classical benchmark problem featuring discontinuous boundary conditions: the lid-driven cavity flow over a wide range of Reynolds numbers. Accurate and oscillation-free solutions are achieved even in the vicinity of the discontinuous top corner boundaries. Lastly, a canonical backward-facing step flow problem is solved, where accuracy and efficiency are demonstrated.
\end{abstract}




\maketitle

\section{Introduction}

The INS equations describe a wide range of viscous flow phenomena encountered in various engineering applications, including aircraft design \cite{jameson1998optimum,changfoot2019hybrid,raj1991computational,malan2002improved1}, automobile engineering \cite{lohner2021overnight,bayraktar2001experimental}, biomedical engineering \cite{zingaro2021hemodynamics,miraucourt2017blood}, and many more. Over the years, numerous numerical techniques have been developed to solve the INS equations in an efficient way. Among them, the CGFEM approach \cite{girault1979finite,gunzburger2012finite,gresho2000incompressible,brooks1982streamline,barragy1997stream,botella1998benchmark,gresho1991some,brezzi2012mixed} has gained increasing attention due to its potential to deliver high-order accurate solutions while maintaining computational efficiency. However, the CGFEM framework remains challenging in several aspects.

One primary concern in simulating incompressible flows lies in enforcing the divergence-free constraint on the velocity field. The commonly adopted primitive variable formulation, based on velocity and pressure - typically results in a saddle-point problem, which requires the use of compatible approximation function spaces for the velocity and pressure fields to satisfy the \emph{Ladyzhenskaya-Babu\v{s}ka-Brezzi} or the \emph{inf-sup} conditions \cite{girault1979finite,gunzburger2012finite,gresho2000incompressible,brezzi2012mixed} in order to ensure the uniqueness of the solution. This criteria gets relaxed for the Least-square-finite-element-method framework. However, it demands a $C^1$ continuous function space to approximate the velocity field. To alleviate the $C^1$ continuity, researchers \cite{jiang1992least,prabhakar2006spectral,prabhakar2007stress} have proposed the use of auxiliary fields. While this approach simplifies the implementation, it significantly increases the computational cost. Alternative strategies have also been explored, such as constructing divergence-conforming grids \cite{evans2013isogeometricSteady,evans2013isogeometricUnsteady,buffa2011isogeometric} that inherently satisfy the continuity equation. However, generating such grids is often complex and nontrivial. Another approach involves relaxing the strict incompressibility constraint by introducing a penalty parameter \cite{hughes1979finite,gunzburger1989iterated,hostos2021solving}, allowing the formulation to accommodate weakly incompressible flows. However, the accuracy of this method is sensitive to the choice of the penalty parameter, and can lead to ill-conditioned matrices and stiffness, if not properly chosen. 

Yet, another alternative is to use the stream function-vorticity-based formulation \cite{girault1979finite,gunzburger2012finite,gresho2000incompressible,barragy1997stream}, which inherently enforces the incompressibility constraint and fits naturally within the $C^0$-continuous finite element framework. However, properly identifying vorticity boundary conditions increases the complexity of implementation. To address this issue, researchers have considered single-field formulations, such as the stream function-based approach \cite{mandal2023weakly,fairag2003numerical,cayco1986finite,cayco1989analysis,gunzburger2012finite,gresho2000incompressible}. Unfortunately, this leads to a biharmonic-type formulation, which requires $C^1$-continuous basis functions. Constructing $C^1$-continuous basis functions is challenging, and researchers have employed elements such as the Argyris variant \cite{gunzburger2012finite}, the Morley triangle \cite{gunzburger2012finite}, the Bogner-Fox-Schmit element \cite{fairag2003numerical}, etc., to meet this requirement. 
In addition, stream function-vorticity and stream function-based approaches are generally applied to two-dimensional domains. Their extension to three-dimensional problems remains rare due to the complexity of implementing boundary conditions. 

Another challenge with CGFEM arises in the presence of discontinuous boundary conditions,  when imposed strongly. Discontinuities in the prescribed boundary data can induce spurious, nonphysical oscillations in the numerical solution, leading to the Gibb's phenomena \cite{cruchaga1997finite,lestandi2018multiple,suman2019grid}. A well-known example is the lid-driven cavity flow problem, where the top boundary is assigned a unit tangential velocity while the remaining boundaries are treated as wall boundaries with zero velocity. This setup introduces discontinuities at the top corners, where the velocity abruptly changes, making the problem difficult to handle with traditional strong enforcement techniques. To mitigate this issue, researchers have proposed smoothing the tangential velocity profile near the corners - commonly using regularization strategies such as hyperbolic tangent functions \cite{prabhakar2006spectral} or quartic polynomial distributions \cite{shen1990numerical} - while setting the corner velocities to zero. Although these methods can alleviate numerical artifacts, they are inherently ad hoc and lack general applicability. A more systematic approach involves the weak imposition of boundary conditions. Popular techniques within the continuous Galerkin finite element framework include the use of Lagrange multipliers \cite{hansbo2005lagrange,burman2010fictitious}, the \emph{Nitsche}-type penalty method \cite{nitsche1971variationsprinzip,hansbo2002unfitted,hansbo2003nitsche,embar2010imposing} and fictitious domain methods \cite{glowinski1994fictitious,rank2011shell}. While the Lagrange multiplier method allows for consistent weak enforcement, it introduces an additional unknown field, significantly increasing computational cost. On the other hand, Nitsche’s method avoids this by incorporating the boundary condition in weak form using a penalty term. However, selecting an appropriate penalty parameter is nontrivial, and a poorly chosen value may lead to ill-conditioned matrices and stiffness.

 To address all limitations above, we employ the SBP-SAT technique \cite{nordstrom2019energy,malan2023sbp,nordstrom2017roadmap,nordstrom2024skew,nordstrom2024explicit} in the CGFEM framework and discretize the INS in skew-symmetric form~\cite{nordstrom2019energy}. To the best of our knowledge \cite{DELREYFERNANDEZ2014171,svard2014review}, the use of the SBP-SAT technique in the CGFEM framework for the nonlinear INS equation is new. The SBP-SAT technique includes a weak imposition of the boundary conditions. For the lid-driven cavity flow problem, this naturally leads to a weighted combination of the discontinuous boundary conditions and therefore requires no special treatment at the top corners. This formulation additionally employs a primitive variable vector which allows for equal-order Lagrange polynomial to be used. The chosen variables in combination with the SAT also remove the null spaces associated with the spatial operator and the related saddle point problem \cite{nordstrom2020spatial}. Stability is guaranteed by proving energy boundness. 

 The paper is organized as follows. The continuous problem and the associated energy analysis are described in Section~\ref{sec:MathematicalFormulations}. The appropriate choice of boundary conditions and its weak imposition technique using SAT are also presented here. Later, the semi-discrete approximation of the continuous problem in the CGFEM framework and its associated energy estimate are discussed in Section~\ref{Sec:Semidiscrete_CGFEM}. Here we also present the discrete energy estimate and prove stability in a similar manner as in the continuous setting. The time discretization using the second-order backward difference technique to obtain the fully discrete scheme is detailed in Section~\ref{Sec:Time_Discretization}. Section~\ref{Sec:Nuemrical_results} presents the numerical solutions for three different problems. First we use the MMS to verify the order of convergence of the formulation. Next, we solve the lid-driven cavity flow problem for a wide range of Reynolds numbers (Re). Finally, we solve the backward-facing step flow problem. Conclusions are presented in Section~\ref{sec:Concluding_remarks}. 

\section{The continuous problem }
\label{sec:MathematicalFormulations}
Let $\Omega\subset\mathbb{R}^{2}$ be the computational domain having boundary $\partial\Omega$. The governing equations for INS in terms of primitive variables in the absence of body force can be written as: 
\begin{equation}
\label{Eq:INS_Nform}
\begin{split}
\partial_{t}u + u\partial_{x}u + v\partial_{y}u + \partial_{x}p -\epsilon\left(\partial_{xx}^{2}+\partial_{yy}^{2}\right)u&=0, \\
\partial_{t}v + u\partial_{x}v + v\partial_{y}v + \partial_{y}p -\epsilon\left(\partial_{xx}^{2}+\partial_{yy}^{2}\right)v&=0, \\
\partial_{x}u + \partial_{y}v &= 0,\\
\end{split}
\end{equation}
where $(u,v)$ defines the velocity components in $(x,y)$ directions, respectively, $p$ the pressure field divided by the constant density, and $\epsilon=1/\mathrm{Re}$ represents the inverse of Reynolds number. The operators $\partial_{t}$, $\partial_{x}$, $\partial_{y}$, $\partial^{2}_{xx}$ and $\partial^{2}_{yy}$ denote temporal and spatial differentiations, respectively. 

To introduce the notations convenient for the upcoming energy analysis, we first define the relevant Sobolev spaces \cite{girault2012finite, brezzi2012mixed}: Let $L^2(\Omega)$ denote the Hilbert space of square-integrable functions on $\Omega$, equipped with the norm $\norm{u}_{L^2(\Omega)}^2=\int_{\Omega}u^{2}~d\Omega$. For an integer $m\geq 1$, the Sobolev space $H^m(\Omega)$  consists of the function along with it's derivatives upto order $m$ belong to $L^2(\Omega)$, i.e., $H^{m}(\Omega)=\{u\in L^2({\Omega}), \partial^{\alpha}u\in L^{2}(\Omega),\; |\alpha|\leq m  \}$, along with the following norm: $\norm{u}_{H^{m}(\Omega)}^{2}=\norm{u}_{L^2{(\Omega)}}^2 + \norm{\partial^{m}u}_{L^{2}(\Omega)}^{2}$. For two-dimensional vector-valued functions, we define $[H^{m}(\Omega)]^2=\{\bm{u}=(u,v)^{T}:u,v\in H^{m(\Omega)} \}$, with norm  $\norm{\bm{u}}_{[H^m{(\Omega)}]^2}^{2}= \norm{u}_{H^m{(\Omega)}}+\norm{v}_{H^m{(\Omega)}} $. 

We now rewrite (\ref{Eq:INS_Nform}) in matrix-vector form, with boundary and initial conditions, which leads to the initial-boundary-value problem~\cite{nordstrom2019energy}: find $\mathbf{w}\in\mathcal{U}$, where $\mathcal{U}:=\{\mathbf{w}\in [H^{2}(\Omega)]^{2}\times H^{1}(\Omega) | H\mathbf{w}=\mathbf{g}~ \text{on}~ \partial \Omega\}$ such that
\begin{align}
\label{Eq:INS_VM_form}
\tilde{I}\partial_{t}\mathbf{w} + \left(A\partial_{x} + B\partial_{y}\right)\mathbf{w} - \epsilon \tilde{I}\left(\partial^{2}_{xx}+\partial^{2}_{yy}\right)\mathbf{w} =\bm{0},& \quad \mathbf{x}\in\Omega,\; t> 0,\nonumber\\
H\mathbf{w} = \mathbf{g},& \quad \mathbf{x}\in\partial\Omega,\; t> 0,\\
\tilde{I}\mathbf{w}=\mathbf{f},& \quad \mathbf{x}\in\Omega,\; t=0,\nonumber
\end{align}
where 
\begin{align}
\label{Eq:MatVec_Op}
\mathbf{w} = \begin{pmatrix}  
u\\
v\\
p
\end{pmatrix},\quad \tilde{I} = \begin{pmatrix}  
1 & 0 & 0\\
0 & 1 & 0\\
0 & 0 & 0
\end{pmatrix},\quad A=\begin{pmatrix} 
u & 0 & 1\\
0 & u & 0\\
1 & 0 & 0
\end{pmatrix},\quad B=\begin{pmatrix} 
v & 0 & 0\\
0 & v & 1\\
0 & 1 & 0
\end{pmatrix}.
\end{align}

where $H^1(\Omega)$ is the closed sobolev subspace.

In (\ref{Eq:INS_VM_form}), $H$ is the boundary operator and the corresponding boundary data is defined by $\mathbf{g}$. The initial data is denoted by $\mathbf{f}$ and is imposed only on velocity components. We will determine the appropriate form of the boundary conditions, defined by $H\mathbf{w}=\mathbf{g}$ in the following energy analysis, to ensure an energy bound.
\subsection{Continuous energy analysis}
\label{Sec:ContEnergyAnalysis}
Following the approach outlined in \cite{nordstrom2017roadmap,nordstrom2019energy}, we formulate an energy-bounded continuous problem, which includes the derivation of boundary conditions. For linear initial boundary value problems, the energy method \cite{kreiss1970initial,straughan2013energy} generates an energy estimate in an $L_{2}$ equivalent norm. However, this technique requires certain symmetry properties of the matrices \cite{nordstrom2024skew}. To address this and simplify the derivation of the energy estimate, the nonlinear advective terms are expressed in skew-symmetric form \cite{nordstrom2019energy,nordstrom2022skew,nordstrom2024skew,nordstrom2006conservative}, assuming that both velocity components are differentiable. Next, we rewrite the momentum balance equation subject to the simplification due to incompressibility, $\partial_{x}A+\partial_{y}B=\left(\partial_{x}u+\partial_{y}v \right)\tilde{I}=0$ as follows :
\begin{equation}
\label{Eq:INS_conservative_NonConservativeForm}
\tilde{I}\partial_{t}\mathbf{w} +\mathcal{D}\left(\mathbf{w}\right)\mathbf{w}=0,
\end{equation}
where the spatial term $(\mathcal{D}(\mathbf{w})\mathbf{w})$ includes both advective and viscous terms, and is defined as
\begin{align}
\label{Eq:Spatial_operator_sbp_sat}
 &\mathcal{D}(\mathbf{w})\mathbf{w}=\Biggl[\frac{1}{2}\Bigl[A\partial_{x} + \partial_{x}A +B\partial_{y} + \partial_{y}B\Bigr] - \epsilon \tilde{I}\left(\partial^{2}_{xx}+\partial^{2}_{yy} \right)\Biggr]\mathbf{w}.
\end{align}

The energy method can now be applied by multiplying  (\ref{Eq:INS_conservative_NonConservativeForm}) with $\mathbf{w}^{T}$ and integrating over $\Omega$ to yield
\begin{align}
\label{Eq:ContEnergyINS_stp01}
\int_{\Omega}\mathbf{w}^{T}\tilde{I}\partial_{t}\mathbf{w}~d\Omega &+\frac{1}{2}\int_{\Omega}\Bigl[\mathbf{w}^{T}A\partial_{x}\mathbf{w} + \mathbf{w}^{T}\partial_{x}\left(A\mathbf{w}\right) + \mathbf{w}^{T}B\partial_{y}\mathbf{w} + \mathbf{w}^{T}\partial_{y}\left(B\mathbf{w} \right) \Bigr]~d\Omega \nonumber\\
&-\epsilon\tilde{I}\int_{\Omega}\left(\mathbf{w}^{T}\partial^{2}_{xx}\mathbf{w}+\mathbf{w}^{T}\partial^{2}_{yy}\mathbf{w} \right)~d\Omega=0.
\end{align}

By applying integration-by-parts, (\ref{Eq:ContEnergyINS_stp01}) is simplified to:
\begin{align}
\label{Eq:ContEnergyINS_step02}
\frac{d}{dt}\norm{\mathbf{w}}^{2}_{\tilde{I}} + 2\epsilon\norm{\grad{\mathbf{w}}}^{2}_{\tilde{I}}=\oint_{\partial\Omega}\textrm{BT}~ds,
\end{align}
where $\norm{\mathbf{w}}^{2}_{\tilde{I}}=\int_{\Omega}\mathbf{w}^{T}\tilde{I}\mathbf{w}~d\Omega$ and $\norm{\grad{\mathbf{w}}}^{2}_{\tilde{I}}=\int_{\Omega}\left((\grad{u})^{T}\grad{u}+(\grad{v})^{T}\grad{v}\right)d\Omega\geq 0$, where $\grad{}=(\partial_{x},\partial_{y})^{T}$. The length of the boundary segment in the surface integral is denoted by $ds=\sqrt{dx^2 + dy^2}$. The boundary term in the surface integral on the right-hand side of (\ref{Eq:ContEnergyINS_step02}) is:
\begin{align} 
\label{Eq:INS_boundaryIntegral}
\textrm{BT} = -u_{n}\left(u_{n}^{2}+u_{t}^{2}+2p  \right)+2\epsilon\left(u_{n}\partial_{n}u_{n}+u_{t}\partial_{n}u_{t}  \right).
\end{align}
The outward unit vector normal to the boundary $\partial\Omega$ is $\mathbf{n}=(n_x,n_y)^{T}$ and the unit tangent vector along the boundary is $\mathbf{t}=(-n_y, n_x )^{T}$. In (\ref{Eq:INS_boundaryIntegral}), the normal and tangential velocity components at the boundaries are represented by $u_{n}=un_{x}+vn_{y}$ and $u_{t}=-un_{y}+vn_{x}$, respectively. Moreover, the normal gradient at the boundary is given by $\partial_{n}=\mathbf{n}\cdot\grad{}$.

Eq.~(\ref{Eq:ContEnergyINS_step02}) describes that the only potential source of energy growth is due to the boundary term $\textrm{BT}$. To ensure stability, $\textrm{BT}$ must be bounded. The detailed analysis of appropriate selections of boundary conditions, including the solid wall cases, is readily available in \cite{nordstrom2019energy}. Extension to general non-homogeneous boundary conditions is given in \cite{nordstrom2024nonlinear,nordstrom2025linear,nordstrom2025open}. The homogeneous boundary conditions lead to the following Proposition:
\begin{proposition}
\label{Proposition:NormBound}
If the  boundary data are homogeneous \cite{nordstrom2025linear} and the boundary term in (\ref{Eq:INS_boundaryIntegral}) satisfies $BT\leq 0$ on $\partial \Omega$, then the solution of (\ref{Eq:INS_conservative_NonConservativeForm}) satisfies the energy bound 
\begin{equation}
\label{Eq:EnergyEstimateStrongForm}
\left(\norm{\mathbf{w}}^{2}_{\tilde{I}}  \right)_{t=T} + 2\epsilon\int_{0}^{T}\norm{\grad{\mathbf{w}}}_{\tilde{I}}^
{2}dt\leq \norm{\mathbf{f}}_{\tilde{I}}^{2}.
\end{equation}
\end{proposition}
\begin{proof}
Time integration of (\ref{Eq:ContEnergyINS_step02}) from 0 to $T$  proves the theorem.
\end{proof}


For the upcoming lid-driven cavity problem, we have limited ourselves to Dirichlet boundary conditions, where we impose $u_{n}=0$ and $u_{t}=1$ at the upper lid and $u_{n}=u_{t}=0$ at the other solid walls. We define the boundary operator $H$ (mentioned in (\ref{Eq:INS_VM_form})) as follows:
\begin{align}
\label{Eq:SolidWallBC_ContProb}
H\mathbf{w}=\begin{pmatrix}n_{x} & n_{y} & 0\\ -n_{y} & n_{x} & 0 \\ 0 & 0 & 0 \end{pmatrix}\begin{pmatrix}u \\ v\\ p  \end{pmatrix}=\begin{pmatrix}u_{n}\\ u_{t}\\0  \end{pmatrix}.
\end{align}
\subsection{Weak imposition of boundary conditions for the continuous problem}
\label{Sec:WeakImpositionBC}
In this section, we impose the boundary conditions defined in (\ref{Eq:INS_VM_form}) and  (\ref{Eq:SolidWallBC_ContProb}) in a weak sense using the SAT method. We thus add the SAT term \cite{nordstrom2017roadmap} to the initial boundary value problem (\ref{Eq:INS_VM_form}) as:
\begin{equation}
\begin{split}
\label{Eq:WeakImposition_INS_SAT}
&\tilde{I}\partial_{t}\mathbf{w}+\mathcal{D}(\mathbf{w})\mathbf{w}=\textrm{SAT}\; \quad \textrm{where}\quad \textrm{SAT}=L\left(\Sigma\left(H\mathbf{w}-\mathbf{g} \right)  \right),\\
&\tilde{I}\mathbf{w}\left(0 \right)=\mathbf{f}.
\end{split}
\end{equation}
Here the spatial operator $\mathcal{D}(\mathbf{w})$ is defined in (\ref{Eq:Spatial_operator_sbp_sat}). The SAT term contains a lifting operator $L\left(\cdot \right)$ \cite{arnold2002unified} which is defined as $\int_{\Omega}\psi^{T}L\left(\phi \right)d\Omega=\oint_{\partial\Omega}\psi^{T}\phi~ds$ for smooth vectors $\psi$ and $\phi$. 

We now apply the energy method (multiply with the solution vector and integrate over the domain) to (\ref{Eq:WeakImposition_INS_SAT}) and obtain
\begin{align}
\label{Eq:WkBCEnergyEtimate}
\frac{d}{dt}\norm{\mathbf{w}}^{2}_{\tilde{I}} + 2\epsilon\norm{\grad{\mathbf{w}}}^{2}_{\tilde{I}}=\oint_{\partial\Omega}\textrm{BC}~ds \;\; \textrm{where}\quad \textrm{BC}=\textrm{BT} + 2\mathbf{w}^{T}\left(\Sigma\left(H\mathbf{w}-\mathbf{g} \right)\right).
\end{align}
The boundary term $\textrm{BT}$ remains identical to (\ref{Eq:INS_boundaryIntegral}). We aim to get an energy bound such that the right-hand-side of (\ref{Eq:WkBCEnergyEtimate}) becomes non-positive. For this reason, we now construct the penalty matrix $\Sigma$ such that it eliminates the indefinite boundary contribution, and for simplicity we use $\mathbf{g}=\mathbf{0}$.

\begin{proposition}
An energy bound for \eqref{Eq:WeakImposition_INS_SAT} is obtained by selecting the penalty matrix $\Sigma = \Sigma_1 + \Sigma_2$, where $\Sigma_1 = R^{T}\Sigma_1'$ and $\Sigma_2 = R^{T}\Sigma_2'$, with
\[
R = \begin{pmatrix} 
n_x & n_y & 0 \\
- n_y & n_x & 0 \\
0 & 0 & 1 
\end{pmatrix}, \quad
\Sigma_1' = \begin{pmatrix}
\dfrac{u_n}{2} & 0 & 0 \\
0 & \dfrac{u_n}{2} & 0 \\
1 & 0 & 0
\end{pmatrix},\quad and \quad
\Sigma_2' = \begin{pmatrix}
- \partial_n^{T} & 0 & 0 \\
0 & - \partial_n^{T} & 0 \\
0 & 0 & 0
\end{pmatrix}.
\]
\end{proposition}
\begin{proof}
Let us start by splitting the penalty matrix $(\Sigma)$ into advection $(\Sigma_{1})$ and diffusion $(\Sigma_{2})$ matrices and rewriting $\textrm{BC}$ in (\ref{Eq:WkBCEnergyEtimate}) as follows:
\begin{align}
\label{Eq:BC_ESTIMATE_SOLIDWALL}
\textrm{BC}=\biggl[-u_{n}\left(u_{n}^{2}+u_{t}^{2}+2p \right)+ 2\mathbf{w}^{T}\Sigma_{1}H\mathbf{w}\biggr] + \biggl[2\epsilon\left(u_{n}\partial_{n}u_{n} + u_{t}\partial_{n}u_{t} \right) + 2\mathbf{w}^{T}\Sigma_{2}H\mathbf{w} \biggr].
\end{align}

The penalty terms $\Sigma_{1}$ and $\Sigma_{2}$ will be selected in such a way that the terms within the brackets are canceled. Let us now consider the first bracket with  $\Sigma_{1}=R^{T}\Sigma_{1}^{\prime}$ where $R\mathbf{w}=\tilde{\mathbf{w}}$ and $\tilde{\mathbf{w}}=(u_{n},u_{t},p)^{T}$ and obtain 
\begin{align}
\label{Eq:penaltyestimate_1stSquareBracket}
-u_{n}\left(u_{n}^{2}+u_{t}^{2}+2p \right) + 2\mathbf{w}^{T}\Sigma_{1}H\mathbf{w}=-u_{n}\left(u_{n}^{2}+u_{t}^{2}+2p \right)+ 2\tilde{\mathbf{w}}^{T}\Sigma_{1}^{\prime}H\mathbf{w}.
\end{align}
To eliminate all terms in the first bracket, we choose (recalling (\ref{Eq:SolidWallBC_ContProb}))
\begin{equation}
\label{Eq:Sigma1_cont}
\Sigma_{1}^{\prime}=\begin{pmatrix}\dfrac{u_n}{2} & 0&0\\ 0 & \dfrac{u_n}{2} & 0\\ 1 &0 & 0 \end{pmatrix}.
\end{equation}
                                                   
\noindent For the second bracket, we set $\Sigma_{2}=R^{T}\epsilon\Sigma_{2}^{\prime}$ and obtain
\begin{align}
\label{Eq:penaltyestimate_2ndSquareBracket}
2\epsilon\left(u_{n}\partial_{n}u_{n} + u_{t}\partial_{n}u_{t} \right) + 2\mathbf{w}^{T}\Sigma_{2}H\mathbf{w}=2\epsilon\left(u_{n}\partial_{n}u_{n} + u_{t}\partial_{n}u_{t} \right) + 2\tilde{\mathbf{w}}^{T}\epsilon\Sigma_{2}^{\prime}H\mathbf{w}.
\end{align}
All terms in (\ref{Eq:penaltyestimate_2ndSquareBracket}) are removed by setting
\begin{equation}
\label{Eq:Sigma2_cont}
\Sigma_{2}^{\prime}=\begin{pmatrix}-\partial_{n}^{T} & 0 &0 \\ 0 & -\partial_{n}^{T}& 0\\ 0 &0 &0 \end{pmatrix}.
\end{equation}
Here the non-conventional operator $\partial_{n}^{T}$ defines the transpose of the normal derivative operator and it acts to the left such that $u\partial_{n}^{T}=\partial_{n} u$. Now the right-hand-side of (\ref{Eq:WkBCEnergyEtimate}) vanishes and its time integration leads to energy boundness such that Proposition~\ref{Proposition:NormBound} holds. This completes the proof. 
\end{proof}
\section{Semi-discrete approximation in the CGFEM framework}
\label{Sec:Semidiscrete_CGFEM}
To solve the problem (\ref{Eq:WeakImposition_INS_SAT}) numerically in the CGFEM framework, we discretize the two-dimensional domain $\Omega\in \mathbb{R}^{2}$ into $n_{\textrm{el}}=M_{\textrm{el}}\times N_{\textrm{el}}$ finite elements such that $\Omega = \cup~\Omega^{e}$. Here, $M_{\textrm{el}}$ and $N_{\textrm{el}}$ denote the number of elements in the $x$- and $y$-directions respectively. The total number of global nodes becomes $M\times N$, where $M=\left(M_{\textrm{el}}\times k+1\right)$ and $N= \left(N_{\textrm{el}}\times k+1\right)$ represent the number of nodes in the corresponding directions and $k$ is the degree of the Lagrange basis used.

We begin by constructing the spatial operators in SBP form at the element level for a single scalar field $\phi$ in a one-dimensional domain, following the methodology described in~\cite{malan2023sbp}. These element-level operators are then assembled to form the corresponding global operators. The two-dimensional operators are subsequently constructed using tensor products of the one-dimensional operators. Finally, the formulation is extended to accommodate the vector field $\mathbf{w}^{T}=(u, v, p)$, making it compatible with the INS equations.
\subsection{Constructions of the SBP operators at the element level for one-dimensional domain}
\label{Sec:Construction}
 For a single element in a one-dimensional domain, we employ the Lagrange basis function $(\mathcal{L})$ to approximate the solution field. As such the trial function $\phi(x,t)$ is approximated by:
 \begin{equation}
\phi^{h}(x,t) = \sum_{i=0}^{k}\mathcal{L}_{i}(x)\Phi_{i}(t)=\bm{\mathcal{L}}^{T}(x)\bm{\Phi}(t),
\end{equation}
where the total number of nodes per element is $\left(k+1\right)$, and $\bm{\mathcal{L}}(x)^{T}=[\mathcal{L}_{0}(x),\mathcal{L}_{2}(x),\ldots,$ $\mathcal{L}_{k}(x)]$ is the set of Lagrange polynomials of degree $k$. $\bm{\Phi}(t)=[\Phi_{0}(t),\Phi_{1}(t),\ldots,\Phi_{k}(t)]$ denotes the set of time-dependent coefficients at node points. 
 
We compute the element mass matrix $(P^{e})$ and weak first derivative operator $(Q_{x}^{e})$ at the element level as (for details see \cite{malan2023sbp}): 
\begin{equation}
\label{Eq:SBP_OPS_1DM01}
P^{e} = \int_{\Omega^{e}}\bm{\mathcal{L}}\bm{\mathcal{L}}^{T}~d\Omega^{e},\quad Q_{x}^{e}=\int_{\Omega^{e}}\bm{\mathcal{L}}\left(\partial_{x}\bm{\mathcal{L}}\right)^{T}~d\Omega^{e},
\end{equation}
where $\int_{\Omega^{e}}(\cdot)~d\Omega^{e}$ defines the integral over the element domain. 
\begin{remark}
The operator $Q_{x}^{e}$ possesses an almost skew-symmetric property, the so-called SBP-property, since
\begin{equation}
\label{Eq:SBP_PropQx}
Q_{x}^{e} = \int_{\Omega^{e}}\bm{\mathcal{L}}\left(\partial_{x}\bm{\mathcal{L}}\right)^{T} d\Omega^{e}= \left. \bm{\mathcal{L}} \bm{\mathcal{L}}^T \right|_{0}^{k}
-\int_{\Omega^{e}}\left(\partial_{x}\mathcal{L}\right)\mathcal{L}^{T}d\Omega^{e}=\tilde{B}^{e}-\left({Q}_{x}^{e}\right)^{T}, 
\end{equation}
where the boundary operator $\tilde{B}^{e}=diag(-1,0,\ldots,0,1)$ contains non-zero values only on the boundary.
\end{remark}
\noindent Similarly, we define the weak second derivative operator $Q_{xx}^{e}$ in SBP form at the element level as:
\begin{equation}
\label{Eq:QxxSBP_OP_MM03}
Q_{xx}^{e}=\int_{\Omega^{e}}\bm{L}\left(\partial_{xx}^{2}\bm{\mathcal{L}}\right)^{T}d\Omega^{e}= \left.\bm{\mathcal{L}}\left(\partial_{x}\bm{\mathcal{L}}\right)^{T} \right|_{0}^{k} - \int_{\Omega^{e}}\left(\partial_{x}\bm{\mathcal{L}} \right)\left(\partial_{x}\bm{\mathcal{L}} \right)^{T}d\Omega^{e}.
\end{equation}

Next, we map each element from the physical domain $x\in [0,l_{x}^{e}]$ (where $l_{x}^{e}$ defines the length of the element) to the parametric domain $\xi\in[-1,1]$ using the Jacobian $J=\partial x/\partial \xi$ and compute the spatial integrals using Gauss-Lobatto (GL) quadrature, e.g., for a given function $f(x)$ the integral is computed as:
\begin{equation}
\int_{-1}^{1}f(x)dx=\sum_{i=0}^{k}f(\xi_{i})\omega_{i}|J|.
\end{equation}
Here $\xi_{i}$ refers to the GL quadrature points, and the corresponding quadrature weights are defined by $\omega_{i}$. For the present study, we also assume that the interior nodes within the elements are collocated at the GL points. Therefore, one can compute $P^{e}$ and $Q_{x}^{e}$ (defined in \ref{Eq:SBP_OPS_1DM01}) as:
\begin{equation}
\begin{split}
P^{e}&\approx \sum_{i=0}^{k} \bm{\mathcal{L}}(\xi_{i})\bm{\mathcal{L}}^{T}(\xi_{i})\omega_{i}|J| = |J|\textrm{diag}[\omega_{0},\omega_{2},\ldots, \omega_{k}],\\
Q_{x}^{e}&\approx\sum_{i=0}^{k}\bm{\mathcal{L}}(\xi_{i})\left(\partial_{\xi}\bm{\mathcal{L}}(\xi_{i})\right)^{T}\omega_{i}. 
\end{split}
\end{equation}
\noindent Once we have $P^{e}$ and $Q_{x}^{e}$, we define the strong-form first derivative SBP operator as:
\begin{equation}
\label{Eq:SBP_1stOrderOp}
D_{x}^{e} = \left(P^{e}\right)^{-1}Q_{x}^{e},
\end{equation}

\noindent The weak form second derivative $Q_{xx}^{e}$ in (\ref{Eq:QxxSBP_OP_MM03}) may now be reformulated as:
\begin{equation}
\label{Eq:QxxSBP_OP_MM04}
Q_{xx}^{e}=\left.\bm{\mathcal{L}}\left(D_{x}^{e}\bm{\mathcal{L}} \right)^{T}\right|_{0}^{k}- \left(D_{x}^{e} \right)^{T}P^{e}D_{x}^{e}=\tilde{B}^{e}D_{x}^{e} - \left(D_{x}^{e} \right)^{T}P^{e}D_{x}^{e}.
\end{equation}

\noindent The corresponding strong-form second derivative SBP operator at the element level is:
\begin{equation}
D_{xx}^{e}=\left(P^{e}\right)^{-1}Q_{xx}^{e}.
\end{equation}
\subsection{Construction of the global matrices}
\label{Sec:GlobalMatrixConstruction}
The single element matrices are next assembled to obtain the global matrices, i.e., the global mass matrix $(P)$, the global $Q_{x}$ matrix and the global $Q_{xx}$ matrix. This process is exemplified by considering a two-element mesh. The global matrix $K$ is now merged as:
\begin{equation}
K = \sum_{e=1}^{n_{\textrm{el}}} K^{e}
=\begin{pmatrix} K^{L}_{00} & \ldots & K^{L}_{0k} & &  0 \\
\vdots & \ddots & \vdots & & \\
 K^{L}_{k0}&\ldots &\left( K^{L}_{kk} + K^{R}_{00}\right)& \ldots & K^{R}_{0k}\\
 & & \vdots&\ddots&\vdots \\
 0& & K^{R}_{k0} & \ldots &K^{R}_{kk}
 \end{pmatrix}.
\end{equation}
where the superscripts $L$ and $R$ denote the left and right element matrices, and the subscript is the associated row and column indices. 
\begin{remark}
After computing the global mass matrix $P$, the global $Q_{x}$, and the global $Q_{xx}$ the global SBP operators for the first and second derivatives are determined as $D_{x}=P^{-1}Q_{x}$ and $D_{xx}=P^{-1}Q_{xx}$, respectively. Note also that the SBP-property in (\ref{Eq:SBP_PropQx}) is preserved for the global operator.
\end{remark}
\subsection{Computations of the spatial operators in two-dimensional domain}
\label{Sec:Extension_Multidimension}
The computed one-dimensional operators are extended to two-dimensions using Kronecker products. Let the subscripts $x$ and $y$ represent the one-dimensional global mass matrices computed in $x$ and $y$ directions, respectively. The corresponding 2D matrices are 
\begin{equation}
\label{Eq:2D_SBP_matrices_scalarField}
\begin{split}
P = {P}_{x}\otimes {P}_{y},\quad D_{x} = {P}_{x}^{-1}{Q}_{x}\otimes I_{N},\quad D_{y} =  I_{M}\otimes {P}_{y}^{-1}{Q}_{y},
\end{split}
\end{equation}
where $I_{M}$ and $I_{N}$ define the identity matrices of dimensions $\left(M\times M\right)$ and $\left(N\times N\right)$, respectively.
\subsection{The discrete system}
We now extend the 2D matrices in (\ref{Eq:2D_SBP_matrices_scalarField}) to the three-field $\mathbf{W}^T=(u,v,p)$ formulation. The corresponding SBP operators are constructed as:
\begin{equation}
\begin{split}
\mathbf{P} &= I_{3}\otimes P,\;~~\mathbf{Q}_{x} = I_{3}\otimes Q_{x},~~~ \mathbf{D}_{x} = I_{3}\otimes D_{x},~\; \mathbf{Q}_{xx} = I_{3}\otimes Q_{xx},~ \mathbf{D}_{xx} = I_{3}\otimes D_{xx},\\
\mathbf{Q}_{y} &= I_{3}\otimes Q_{y},~ \mathbf{D}_{y} = I_{3}\otimes D_{y},,~ \mathbf{Q}_{yy} = I_{3}\otimes Q_{yy},~ \mathbf{D}_{yy} = I_{3}\otimes D_{yy}.
\end{split}
\end{equation}

Therefore, with the introduction of the SBP operators, we may define the corresponding variational problem for the INS (\ref{Eq:WeakImposition_INS_SAT}) by pre-multiplying with the weight function and integrating over $\Omega$ as follows: find $\mathbf{W}\in \mathcal{V}$, where $\mathcal{V}:=\{\mathbf{W}\in [H^{1}(\Omega)]^{2}\times L^{2}(\Omega) \}$ such that
\begin{equation}
\label{Eq:SemiDiscreteFormFE_INS_SAT}
\tilde{\mathbf{I}}\mathbf{P}\partial_{t}\mathbf{W}+\dfrac{1}{2}\left(\mathbf{A}\mathbf{Q}_{x}+\mathbf{Q}_{x}\mathbf{A}+\mathbf{B}\mathbf{Q}_{y}+\mathbf{Q}_{y}\mathbf{B}\right)-\epsilon\tilde{\mathbf{I}}\left(\mathbf{Q}_{xx} + \mathbf{Q}_{yy} \right)=\sum_{l\in\{n,s,e,w\}}\Sigma^{l}\left(I_{3}\otimes\mathbb{P}^{l}\right)\left(\bm{\mathcal{H}}^{l}\mathbf{W}-\bm{\mathcal{G}}^{l}  \right).
\end{equation}
The discrete version of the continuous INS problem in (\ref{Eq:WeakImposition_INS_SAT}) using (\ref{Eq:SemiDiscreteFormFE_INS_SAT}) now reads:
\begin{equation}
\label{Eq:Semi_discreteApprox_mod01}
\begin{split}
&\tilde{\mathbf{I}}\partial_{t}\mathbf{W} +\bm{\mathcal{D}}\left(\mathbf{W} \right)\mathbf{W}=\mathbb{SAT}\quad\textrm{where}\quad\mathbb{SAT}=\sum_{l\in\{n,s,e,w\}}\mathbf{P}^{-1}\Sigma^{l}\left(I_{3}\otimes\mathbb{P}^{l}\right)\left(\bm{\mathcal{H}}^{l}\mathbf{W}-\bm{\mathcal{G}}^{l}  \right),\\
&\tilde{\mathbf{I}}\mathbf{W}\left(0\right) = \bm{\mathcal{F}},
\end{split}
\end{equation}
where the operator $\bm{\mathcal{D}}(\mathbf{w})$ is the discrete counterpart of the operator $\mathcal{D}(\mathbf{w})$ in (\ref{Eq:Spatial_operator_sbp_sat}) and (\ref{Eq:WeakImposition_INS_SAT}) given by:
\begin{equation}
\bm{\mathcal{D}}\left(\mathbf{W} \right)=\dfrac{1}{2}\left(\mathbf{A}{\mathbf{D}_{x}} + {\mathbf{D}_{x}}\mathbf{A} + \mathbf{B}{\mathbf{D}_{y}} + {\mathbf{D}_{y}}\mathbf{B}\right) - \epsilon\tilde{\mathbf{I}}\left({\mathbf{D}^{2}_{xx}} + {\mathbf{D}^{2}_{yy}}  \right).
\end{equation}
Further, the block diagonal matrices $\tilde{\mathbf{I}}$, $\mathbf{A}$, and $\mathbf{B}$ which commute with $\mathbf{P}$ are defined as:
\begin{align}
\tilde{\mathbf{I}} = \begin{pmatrix}\mathbf{I} & \bm{0} &\bm{0} \\ \bm{0} & \mathbf{I} & \bm{0} \\ \bm{0} & \bm{0} &\bm{0}   \end{pmatrix};\quad \mathbf{A}=\begin{pmatrix} \textrm{diag}(\mathbf{u}) & \bm{0} & \mathbf{I} \\ \bm{0}&\textrm{diag}(\mathbf{u}) & \bm{0} \\ \mathbf{I} & \bm{0} & \bm{0} \end{pmatrix};\quad \mathbf{B} = \begin{pmatrix}\textrm{diag}(\mathbf{v})& \bm{0} &\bm{0} \\ \bm{0} & \textrm{diag}(\mathbf{v}) & \mathbf{I} \\ \bm{0} & \mathbf{I} & \bm{0}  \end{pmatrix},
\end{align}
where $\mathbf{I}$ and $\bm{0}$ represent the identity and zero matrices of size $(M\times N)\times (M\times N)$.

\begin{remark}
The $\mathbb{SAT}$ term in the semi-discrete formulation (\ref{Eq:Semi_discreteApprox_mod01}) is related to the continuous formulation (\ref{Eq:WeakImposition_INS_SAT}) with $\mathbf{P}^{-1}$ as the discrete lifting operator.
\end{remark}

In (\ref{Eq:Semi_discreteApprox_mod01}), the operator $\mathbb{P}^{l}$ is the global mass matrix restricted to the boundary segment $l$ defined as:
\begin{equation}
\mathbb{P}^{l\in\{n,s,e,w\}} =\begin{cases}
{P}_{x} \otimes E_{N}, & \text{for north boundary}, \quad \text{where}\;\; E_{N} = \operatorname{diag}(0, \ldots, 0, 1), \\
P_{x} \otimes E_{1}, & \text{for south boundary}, \quad \text{where}\;\;E_{1} = \operatorname{diag}(1, 0, \ldots, 0), \\
E_{M} \otimes {P}_{y}, & \text{for east boundary}, \quad \;\text{where}\;\; E_{M} = \operatorname{diag}(0, \ldots, 0, 1), \\
E_{1} \otimes {P}_{y}, & \text{for west boundary}, \quad \;\text{where} \;\;E_{1} = \operatorname{diag}(1, 0, \ldots, 0).
\end{cases}
\end{equation}
The operators $\bm{\mathcal{H}}^{l}$ discretely approximate the continuous boundary operator $H$ (see (\ref{Eq:SolidWallBC_ContProb})) and $\bm{\mathcal{G}}^{l}$ defines the boundary data prescribed at the boundary nodes. The vector $\bm{\mathcal{F}}$ contains the initial data prescribed on the node points. The penalty matrix $\Sigma^{l}$ mimics the penalty matrix of the continuous problem (\ref{Eq:WeakImposition_INS_SAT}). 

In this paper, we limit ourselves to rectangular domains and orthogonal meshes. Before deriving the discrete energy estimate, we need the unit normals along the boundary segment $l$ as:  
\begin{equation}
\left(\mathbf{N}_{x}^{l},\mathbf{N}_{y}^{l}\right)=\begin{cases}
\left(\bm{0},I_{M}\otimes E_{N}  \right) & \text{for north boundary},\\
\left(\bm{0},I_{M}\otimes -E_{1}  \right) & \text{for south boundary},\\
\left(E_{M}\otimes I_{N},\bm{0}   \right) & \text{for east boundary},\\
\left(-E_{1}\otimes I_{N},\bm{0} \right) & \text{for west boundary}.
\end{cases}
\end{equation}

\noindent Next, we rewrite the boundary conditions for the continuous problem (defined in (\ref{Eq:SolidWallBC_ContProb})) in a discrete sense as:
\begin{equation}
\bm{\mathcal{H}}\mathbf{W}=\begin{pmatrix}\mathbf{N}_{x}^{l} & \mathbf{N}_{y}^{l} & \bm{0} \\
-\mathbf{N}_{y}^{l} & \mathbf{N}_{x}^{l} & \bm{0}\\ \bm{0} & \bm{0} & \bm{0}\end{pmatrix}\begin{pmatrix}\mathbf{u}\\\mathbf{v}\\ \mathbf{p}  \end{pmatrix}=\begin{pmatrix}\mathbf{U}_{n}^{l}\\ \mathbf{U}_{t}^{l} \\ \bm{0} \end{pmatrix}.
\end{equation}
where $\mathbf{U}_{n}^{l}$ and $\mathbf{U}_{t}^{l}$ are the normal and tangential velocity components defined at the boundary nodes.
\begin{remark}
Since $\mathbf{u}, \mathbf{v} \in H^1(\Omega)$, their traces on the boundary belong to 
$ H^{1/2}(\partial\Omega)$. Consequently the boundary operators involving normal and tangential components are defined in the weak sense.    
\end{remark}
The norm in $H^{1/2}(\partial\Omega)$ is defined by $\norm{u}_{H^{1/2}(\partial\Omega)}^{2}=\norm{u_{L^2{(\partial\Omega)}}}^{2}+\int_{\partial\Omega}\int_{\partial\Omega}\frac{|u(x)-u(y)|^{2}}{|x-y|^2}ds_{x} ds_{y} $, for any two arbitrary points $x,y$ on $\partial\Omega$, where $ds_x$ and $ds_y$ denote the corresponding infinitesimal length elements at these points.
\section{The semi-discrete energy estimate}
Following \cite{nordstrom2017roadmap}, we now mimic the continuous energy method in Section~\ref{Sec:ContEnergyAnalysis} by multiplying (\ref{Eq:Semi_discreteApprox_mod01}) with $\mathbf{W}^{T}\mathbf{P}$, adding its transpose, and utilizing the SBP properties in (\ref{Eq:SBP_PropQx}) and (\ref{Eq:QxxSBP_OP_MM04}) to obtain:
\begin{equation}
\label{Eq:Semidiscrete_EnergyEq1}
\dfrac{d}{dt}\norm{\mathbf{W}}^{2}_{\tilde{\mathbf{I}}\mathbf{P}} + 2\epsilon\left(\norm{\mathbf{D}_{x}\mathbf{W}}^{2}_{\tilde{\mathbf{I}}\mathbf{P}} + \norm{\mathbf{D}_{y}\mathbf{W}}^{2}_{\tilde{\mathbf{I}}\mathbf{P}}\right)=\mathbb{BC}\;\textrm{where}\;\mathbb{BC}=\mathbb{BT} + \sum_{l\in\{n,e,s,w\}}2\mathbf{W}^{T} \Sigma^{l}\left(I_{3}\otimes\mathbb{P}^{l}\right)\left(\bm{\mathcal{H}}^{l}\mathbf{W}-\bm{\mathcal{G}}^{l}  \right).
\end{equation}
Here $\norm{\mathbf{W}}^{2}_{\tilde{\mathbf{I}}\mathbf{P}}=\mathbf{W}^{T}\tilde{\mathbf{I}}\mathbf{P}\mathbf{W}$ and the dissipation $\norm{\mathbf{D}_{x}\mathbf{W}}^{2}_{\tilde{\mathbf{I}}\mathbf{P}}+\norm{\mathbf{D}_{y}\mathbf{W}}^{2}_{\tilde{\mathbf{I}}\mathbf{P}}=$$\left(\mathbf{D}_{x}\mathbf{W}\right)^{T}\tilde{\mathbf{I}}\mathbf{P}\left(\mathbf{D}_{x}\mathbf{W} \right)+ $ $\left(\mathbf{D}_{y}\mathbf{W}\right)^{T}$$\tilde{\mathbf{I}}\mathbf{P}$ $\left(\mathbf{D}_{y}\mathbf{W} \right)$ is strictly non-negative. The boundary term $\mathbb{BT}$ corresponding to (\ref{Eq:INS_boundaryIntegral}) on the RHS of (\ref{Eq:Semidiscrete_EnergyEq1}) is: 
\begin{equation}
\begin{split}
\mathbb{BT}=&-\sum_{l\in\{n,e,s,w \}}\Biggl[\left(\mathbf{U}_{n}^{l}\right)^{T}\left(I_{3}\otimes\mathbb{P}^{l}\right)\operatorname{diag}\left(\mathbf{U}_{n}^{l} \right)\mathbf{U}_{n}^{l} + \left(\mathbf{U}_{t}^{l}\right)^{T}\left(I_{3}\otimes\mathbb{P}^{l}\right)\operatorname{diag}\left(\mathbf{U}_{n}^{l} \right)\mathbf{U}_{t}^{l} + 2\left(\mathbf{U}_{n}^{l}\right)^{T}\left(I_{3}\otimes\mathbb{P}^{l}\right)\mathbf{p}\\
&\qquad\quad-2\epsilon\left(\left(\mathbf{U}_{n}^{l}\right)^{T}\left(I_{3}\otimes\mathbb{P}^{l}\right)\mathbf{D}_{n}^{l}\mathbf{U}_{n}^{l} + \left(\mathbf{U}_{t}^{l}\right)^{T}\left(I_{3}\otimes\mathbb{P}^{l}\right)\mathbf{D}_{n}^{l}\mathbf{U}_{t}^{l} \right)\Biggr],
\end{split}
\end{equation}
where $\mathbf{D}_{n}^{l}=\mathbf{N}_{x}^{l}\mathbf{D}_{x} + \mathbf{N}_{y}^{l}\mathbf{D}_{y}$ represents the normal derivative at the boundary segment $l$. 

\begin{proposition}
The energy rate in (\ref{Eq:Semidiscrete_EnergyEq1}) leads for energy stability if $\mathbb{BC}$ is non-positive for $\bm{\mathcal{G}}=\bm{0}$.
\end{proposition}
\begin{proof}
Let $\mathbb{BC}$ be non-positive, then for $\bm{\mathcal{G}}=\bm{0}$, the time integration of (\ref{Eq:Semidiscrete_EnergyEq1}) yields the energy estimate:
\begin{equation}
\label{Eq:Semidiscrete_EnergyEstimate}
\left(\norm{\mathbf{W}}^{2}_{\tilde{\mathbf{I}}\mathbf{P}}\right)_{t=T} + 2\epsilon\int_{0}^{T}\left(\norm{\mathbf{D}_{x}\mathbf{W}}^{2}_{\tilde{\mathbf{I}}\mathbf{P}} + \norm{\mathbf{D}_{y}\mathbf{W}}^{2}_{\tilde{\mathbf{I}}\mathbf{P}} \right)dt\leq \norm{\bm{\mathcal{F}}}_{\tilde{\mathbf{I}}\mathbf{P}}.
\end{equation}
This proves the theorem. 
\end{proof}
Eq.~(\ref{Eq:Semidiscrete_EnergyEstimate}) represents the semi-discrete energy estimate of the continuous version (\ref{Eq:EnergyEstimateStrongForm}). 
\subsection{Weak imposition of the boundary conditions for the discrete problem}
\label{Sec:WeakImpositionBC_DiscreteProb}
We now mimic the weak imposition of the boundary conditions in the continuous problem, which leads to Proposition~\ref{Proposition:SemiDiscreteEnergystability}. 
\begin{proposition}
\label{Proposition:SemiDiscreteEnergystability}
The semi-discrete formulation \eqref{Eq:Semidiscrete_EnergyEq1} satisfies an energy estimate if the penalty matrix is chosen as 
$
\bm{\Sigma}^{l} = (\mathbf{R}^{l})^{T} \bigl(\bm{\Sigma}_{1}^{\prime l} + \epsilon \bm{\Sigma}_{2}^{\prime l}\bigr),
$
where
\begin{equation}
\label{Eq:PenalMat_Discrete}
\mathbf{R}^{l} =
\begin{pmatrix}
\mathbf{N}_{x}^{l} & \mathbf{N}_{y}^{l} & \bm{0} \\
- \mathbf{N}_{y}^{l} & \mathbf{N}_{x}^{l} & \bm{0} \\
\bm{0} & \bm{0} & \mathbf{I}
\end{pmatrix}, \quad
\bm{\Sigma}_{1}^{\prime l} =
\begin{pmatrix}
\dfrac{1}{2}\mathrm{diag}(\mathbf{U}_{n}^{l}) & \bm{0} & \bm{0} \\
\bm{0} & \dfrac{1}{2}\mathrm{diag}(\mathbf{U}_{n}^{l}) & \bm{0} \\
\mathbf{I} & \bm{0} & \bm{0}
\end{pmatrix}, \quad and \quad
\bm{\Sigma}_{2}^{\prime l} =
\begin{pmatrix}
- (\mathbf{D}_{n}^{l})^{T} & \bm{0} & \bm{0} \\
\bm{0} & - (\mathbf{D}_{n}^{l})^{T} & \bm{0} \\
\bm{0} & \bm{0} & \bm{0}
\end{pmatrix}.
\end{equation}
\end{proposition}

\begin{proof}
We use $\mathbf{R}\mathbf{W}=\tilde{\mathbf{W}}$ and write $\mathbb{BC}$ in (\ref{Eq:Semidiscrete_EnergyEq1}) as:
\begin{equation}
\label{Eq:Discrete_BC01}
\begin{split}
\mathbb{BC}= &\sum_{l\in\{n,e,s,w \}}\Biggl[-\left(\mathbf{U}_{n}^{l}\right)^{T}\left(I_{3}\otimes\mathbb{P}^{l}\right)\operatorname{diag}\left(\mathbf{U}_{n}^{l} \right)U_{n}^{l} + \left(U_{t}^{l}\right)^{T}\mathbb{P}^{l}\operatorname{diag}\left(\mathbf{U}_{n}^{l} \right)\mathbf{U}_{t}^{l} + 2\left(\mathbf{U}_{n}^{l}\right)^{T}\left(I_{3}\otimes\mathbb{P}^{l}\right)\mathbf{p}\\
&+2\tilde{\mathbf{W}}^{T}\Sigma_{1}^{\prime l}\left(I_{3}\otimes\mathbb{P}^{l}\right)\left(\bm{\mathcal{H}}^{l}\mathbf{W}-\bm{\mathcal{G}}^{l}  \right)\Biggr]+2\epsilon\sum_{l\in\{n,e,s,w \}}\Biggl[\left(\mathbf{U}_{n}^{l}\right)^{T}\left(I_{3}\otimes\mathbb{P}^{l}\right)\mathbf{D}_{n}^{l}\mathbf{U}_{n}^{l} + \left(\mathbf{U}_{t}^{l}\right)^{T}\left(I_{3}\otimes\mathbb{P}^{l}\right)\mathbf{D}_{n}^{l}\mathbf{U}_{t}^{l} \\
&+ \tilde{\mathbf{W}}^{T}\Sigma_{2}^{\prime l}\left(I_{3}\otimes\mathbb{P}^{l}\right)\left(\bm{\mathcal{H}}^{l}\mathbf{W}-\bm{\mathcal{G}}^{l}  \right) \Biggr].
\end{split}
\end{equation}
The matrices $\mathbf{R}^{l}$, $\bm{\Sigma}^{\prime l}_{1}$ and $\bm{\Sigma}^{\prime l}_{2}$ in (\ref{Eq:PenalMat_Discrete}) correspond to their continuous counterparts $R$, $\Sigma_{1}^{\prime}$ and $\Sigma_{2}^{\prime}$ in (\ref{Eq:Sigma1_cont})-(\ref{Eq:Sigma2_cont}). Substitution of the penalty matrices defined in (\ref{Eq:PenalMat_Discrete}) into (\ref{Eq:Discrete_BC01}) leads to $\mathbb{BC}=\bm{0}$ in the absence of boundary data. Therefore, the energy estimate (\ref{Eq:Semidiscrete_EnergyEstimate}) follows, and semi-discrete energy stability is proved.
\end{proof}
\section{Time discretizations}
\label{Sec:Time_Discretization}
Here we briefly discuss the temporal discretizations of the semi-discrete problem in (\ref{Eq:Semi_discreteApprox_mod01}). The commonly employed \cite{BurdenFairs2010,tafti1996comparison,baker1982higher,emmrich2004error,wang2012efficient} second-order Backward-Difference (BDF2) implicit time integration scheme with uniform temporal step size $\Delta t$ is employed. We denote the unknown solution fields $\mathbf{W}$ at three consecutive times $t$, $(t+\Delta t)$ and $(t+2\Delta t)$ by $\mathbf{W}^{t}$, $\mathbf{W}^{t+1}$ and $\mathbf{W}^{t+2}$, respectively. The discrete problem to be solved for $\mathbf{W}^{t+2}$ is:
\begin{equation}
\label{Eq:FullDiscrete_nonLin}
\bm{\mathcal{R}}\left(\mathbf{W}^{t+2}  \right)=\tilde{\mathbf{I}}\dfrac{3\mathbf{W}^{t+2}-4\mathbf{W}^{t+1}+\mathbf{W}^{t}}{2\Delta t}+\tilde{\bm{\mathcal{D}}}\left(\mathbf{W}^{t+2}\right)\mathbf{W}^{t+2}=0,
\end{equation}
where $\bm{\mathcal{R}}\left(\mathbf{W}^{t+2}\right)$ is the residual. The operator $\tilde{\bm{\mathcal{D}}}$ accounts for the spatial operator and the $\mathbb{SAT}$ term in (\ref{Eq:Semi_discreteApprox_mod01}).  

\noindent Due to the nonlinearity of (\ref{Eq:FullDiscrete_nonLin}), Newton's method is employed:
\begin{equation}
\label{Eq:FullyDiscreteApprox_lin}
\mathbf{W}^{t+2}=\mathbf{W}^{t+1}-\mathbf{J}^{-1}\bm{\mathcal{R}}\left(\mathbf{W}^{t+1},\mathbf{W}^{t} \right),
\end{equation}
where $\mathbf{J}$ is the Jacobian matrix which is computed exactly following \cite{nordstrom2024explicit}. We solve the linearized problem (\ref{Eq:FullyDiscreteApprox_lin}) iteratively until $\norm{\mathbf{W}^{t+1}-\mathbf{W}^{t}}_{\mathbf{P}}^{2}< tol$, with $\textrm{tol}=10^{-10}$.  
\section{Numerical results}
\label{Sec:Nuemrical_results}
Our code is developed in the MATLAB environment (MATLAB-R2024b). We start with the convergence study using the MMS technique \cite{nordstrom2024explicit,roache2002code} to verify the implementation, and then model the canonical lid-driven cavity and backward-facing step flow problems.
\subsection{Grid convergence study}
\label{Sec:Error_convStudy}
We start with the grid convergence study on the 2D domain $\Omega=[0,1]\times[0,1]$. The MMS solution is:
\begin{equation}
\label{Eq:MMS_INS}
\begin{split}
&u(x,y,t)= 1+0.1\sin(3\pi x-0.40t)\sin(3\pi y- 0.40t),\\
&v(x,y,t)= 1+0.1\cos(3\pi x-0.40t)\cos(3\pi y-0.40t),\\
&p(x,y,t)=\cos(3\pi x-0.40t)\cos(3\pi y-0.40t).
\end{split}
\end{equation}
Substitutions of these analytical solutions into (\ref{Eq:INS_VM_form}) generate initial conditions, boundary conditions and forcing functions. Four uniform meshes consisting of $13\times 13$, $25\times 25$, $37\times 37$ and $49\times 49$ grid points are employed. For each mesh, Lagrange basis functions of degree $k=1$ to $k=4$ are considered. The order of convergence is evaluated with respect to the $\mathbf{P}$-norm of the error: $E=\norm{\mathbf{W}-\mathbf{W}^{h}}_{\mathbf{P}}$$=\sqrt{(\mathbf{W}-\mathbf{W}^{h})^{T}\mathbf{P}(\mathbf{W}-\mathbf{W}^{h})}$. The exact solution $\mathbf{W}$ is obtained by prescribing the analytical solution (\ref{Eq:MMS_INS}) at the nodes and $\mathbf{W}^{h}$ is the numerical one. The simulations are performed using the  time integration scheme in Section~\ref{Sec:Time_Discretization}. A fixed time step of $\Delta t=$ 6.4e-05 is used to minimize temporal error, and the order of convergence is computed at $t=0.4$ seconds. The viscous coefficient for this study is set to $\epsilon=0.1$. To compute the order of convergence between two successive meshes with node counts $N_1$ and $N_2$ in each parametric direction, we use the following formula:
\begin{equation}
\mathcal{O}(E) = \frac{\log_{10}(E_2) - \log_{10}(E_1)}{\log_{10}(N_1) - \log_{10}(N_2)}.
\end{equation} 
\begin{table}
\centering
\caption{Error norms and convergence rates for Lagrange basis functions of degrees 1 to 4}
\label{tab:error_norms_P1P2P3P4}
\resizebox{1.0\textwidth}{!}{%
\begin{tabular}{ccccccccc}
\hline
\multirow{2}{*}{No of nodes} & \multicolumn{4}{c}{$k=1$}                                       & \multicolumn{4}{c}{$k=2$}                                       \\ \cline{2-9} 
                             & ${\norm{u-u^h}}_{P}$ & $\mathcal{O}(E)$ & ${\norm{v-v^h}}_{P}$ & $\mathcal{O}(E)$  & ${\norm{u-u^h}}_{P}$ & $\mathcal{O}(E)$  & ${\norm{v-v^h}}_{P}$ & $\mathcal{O}(E)$  \\ \hline
$13\times13$                 & 1.019e-02	    &  --    		& 1.051e-02 	      & --   		& 6.148e-03 		& --   		     & 6.306e-03 	    & --   \\
$25\times25$                 & 2.119e-03 	    &  2.40  		& 2.122e-03 	      & 2.45 		& 9.233e-04 		& 2.90 		     & 9.528e-04 	    & 2.89 \\
$37\times37$                 & 9.763e-04 	    &  1.98  		& 9.672e-04 	      & 1.98 		& 3.778e-04 		& 2.28 		     & 3.911e-04 	    & 2.27 \\
$49\times49$                 & 5.668e-04 	    &  1.94  		& 5.676e-04 	      & 1.90 		& 2.067e-04 		& 2.15 		     & 2.141e-04 	    & 2.14 \\ 
\hline
\multirow{2}{*}{No of nodes} & \multicolumn{4}{c}{$k=3$}                                       & \multicolumn{4}{c}{$k=4$}                                       \\ \cline{2-9} 
                             & ${\norm{u-u^h}}_{P}$ & $\mathcal{O}(E)$ & ${\norm{v-v^h}}_{P}$ & $\mathcal{O}(E)$  & ${\norm{u-u^h}}_{P}$ & $\mathcal{O}(E)$  & ${\norm{v-v^h}}_{P}$ & $\mathcal{O}(E)$  \\ \hline
$13\times13$                 &   2.869e-03  	    &    -    		&  3.097e-03  		&    -   	&   8.566e-04  		&   -    		&  8.368e-04  		&    -    \\
$25\times25$                 &   9.539e-05  	    &    5.20 		&  9.784e-05  		&    5.28	&   2.863e-05  		&   5.20 		&  2.692e-05  		&    5.26 \\
$37\times37$                 &   1.702e-05  	    &    4.40 		&  1.726e-05  		&    4.43	&   4.770e-06  		&   4.57 		&  4.826e-06  		&    4.38\\
$49\times49$                 &   5.345e-06  	    &    4.12 		&  5.257e-06  		&    4.23	&   1.550e-06  		&   4.00 		&  1.594e-06  		&    3.94 \\ 
\hline
\end{tabular}%
}
\end{table}
\FloatBarrier

From the above Table~\ref{tab:error_norms_P1P2P3P4} it is clear that the order of convergence becomes $(k+1)$ for odd degree and $k$ for even degree Lagrange polynomials. 
We conclude that our implementation is correct.

\subsection{Lid-driven cavity flow problem}
\label{Sec:LDC_PROBLEM}
We consider the canonical lid-driven cavity flow \cite{merrick2018novel,malan2011artificial,oxtoby2012matrix,mandal2023weakly} to assess the performance of our formulation, as shown in the Figure \ref{Fig:DrivenCavityGeom}. 
Despite its apparent simplicity in geometry, the lid-driven cavity problem is challenging due to the presence of velocity jumps in the top two corners. This will be a key test for evaluating the versatility of the SAT boundary conditions. 

In our numerical experiments, we employ a mesh that is refined near the boundaries to improve efficiency and accuracy. This is achieved by stretching the element edge lengths in the $x$- and $y$-directions over the domain $[0,1]$. For a mesh with $M_{\text{el}} \times N_{\text{el}}$ elements, the stretched element edges are defined as \cite{botella1998benchmark}:
\begin{equation}
x_{\text{edge}}(i) = (1 - \cos\pi i / (M - 1))/2,\;\; y_{\text{edge}}(j) = (1 - \cos\pi j / (N - 1))/2, \; i = 0,\ldots, M-1, \;j = 0,\ldots, N-1, \
\end{equation}
To explain the construction of the two-dimensional stretched mesh, we first consider the one-dimensional case. Within each element, we place $(k+1)$ GL nodes, denoted by $\xi_j$, $j=0,\ldots,k$, which are defined on the reference interval $[-1,1]$. Let $x_L = x_{\text{edge}}(i)$ and $x_R = x_{\text{edge}}(i+1)$ be the left and right endpoints of the $i$-th element. The local GL nodes are mapped to the physical coordinates using an affine transformation:
\begin{equation}
x_{\text{local}} = \frac{x_R - x_L}{2} \, \xi_j + \frac{x_R + x_L}{2}, \quad j = 0,\ldots,k.
\end{equation}
A globally continuous stretched grid is then constructed by assembling these transformed nodes across all elements. The same procedure is applied in the $y$-direction. However, because the resulting mesh is non-uniform due to stretching, a coordinate transformation is required to accurately compute SBP operators on such grids. We adopt the methodology described in \cite{aalund2019encapsulated} to compute both the first-order SBP derivative operators and the Laplacian operator associated with the viscous terms, utilizing Definition 3 and Corollary 1, respectively. We use $4^{\text{th}}$-order Lagrange polynomial basis functions on a mesh consisting of $25 \times 25$ elements, which results in a total of $101 \times 101$  node points across the domain. The resulting stretched mesh is shown in Figure ~\ref{Fig:DrivenCavityGeom}.

We initialize all field variables to zero and iteratively solve using (\ref{Eq:FullDiscrete_nonLin}) and (\ref{Eq:FullyDiscreteApprox_lin}) to obtain the steady-state solution by marching with $\Delta t = 0.1$. The problem is solved for a wide range of Reynolds numbers, varying from $\textrm{Re} = 100$ to $\textrm{Re} = 10,000$, to investigate the behavior of the internal flow of a weakly viscous fluid. The computed horizontal and vertical velocities along the vertical and horizontal centerlines of the lid-driven cavity show excellent agreement with the benchmark results \cite{ghia1982high} in Figure~\ref{Figs:LDC_y_u_x_v_comp_study}.
\FloatBarrier
\begin{figure}
\centering
\begin{subfigure}[t]{0.465\textwidth}
    \centering
        \includegraphics[trim={0.0cm 0.0cm 0.0cm 0.0cm}, clip=true,width=1.0\textwidth]{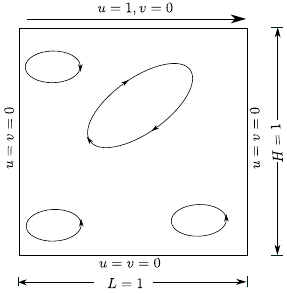}
        \caption*{}
    \end{subfigure}%
\centering
		\begin{subfigure}[t]{0.465\textwidth}
    \centering
        \includegraphics[trim={0.0cm 0.0cm 0.0cm 0.0cm}, clip=true,width=1.0\textwidth]{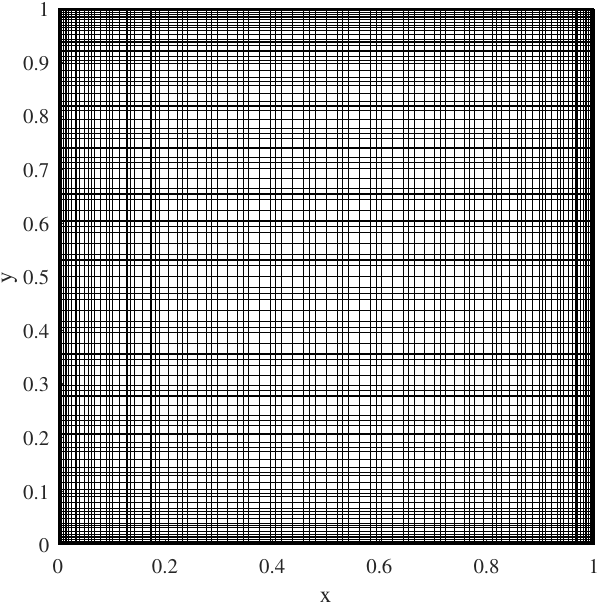}
        \caption*{}
    \end{subfigure}
    \caption{The driven cavity problem geometry (left) and (right) discretization details using $101\times 101$ grid points}
    \label{Fig:DrivenCavityGeom}
\end{figure} 
\begin{figure}
\centering
		\begin{subfigure}[t]{0.32\textwidth}
    \centering
        \includegraphics[trim={0.0cm 0.0cm 0.0cm 0.0cm}, clip=true,width=1.0\textwidth]{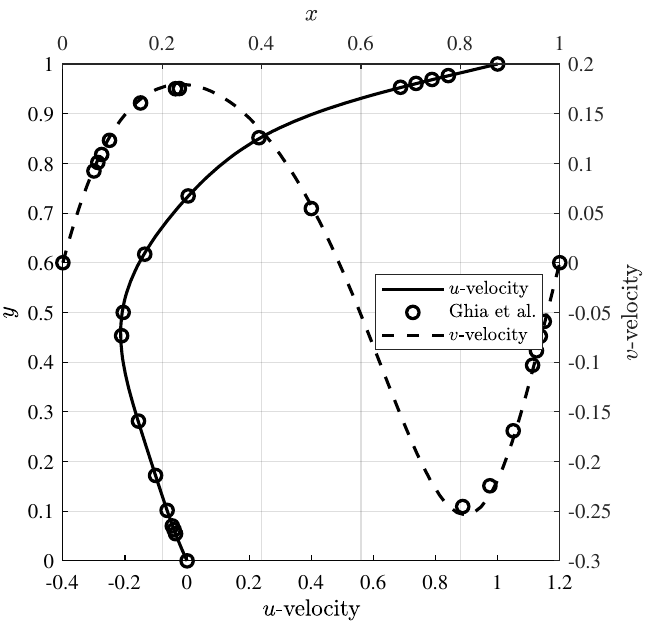}
        \caption{$\textrm{Re}=100$}\label{fig:y_u_x_v_Re_100}
    \end{subfigure}%
    \centering
		\begin{subfigure}[t]{0.32\textwidth}
    \centering
        \includegraphics[trim={0.0cm 0.0cm 0.0cm 0.0cm}, clip=true,width=1.0\textwidth]{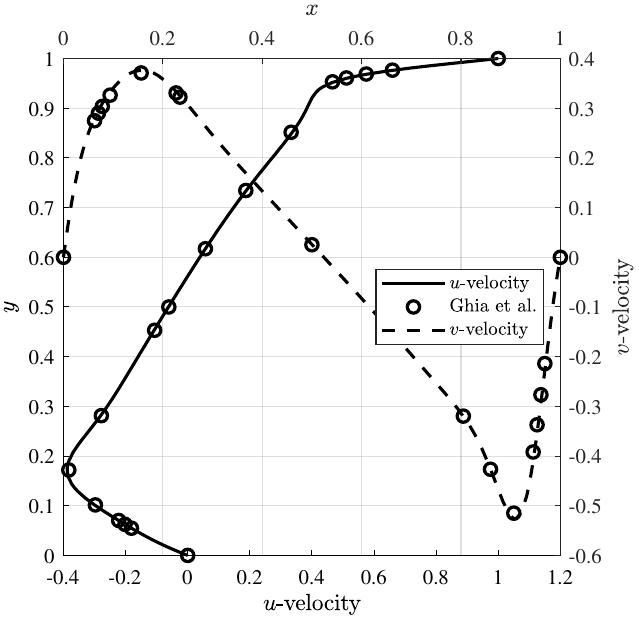}
        \caption{$\textrm{Re}=1,000$}\label{fig:y_u_x_v_Re_1000}
    \end{subfigure}%
  \begin{subfigure}[t]{0.32\textwidth}
    \centering
        \includegraphics[trim={0.0cm 0.0cm 0.0cm 0.0cm}, clip=true,width=1.0\textwidth]{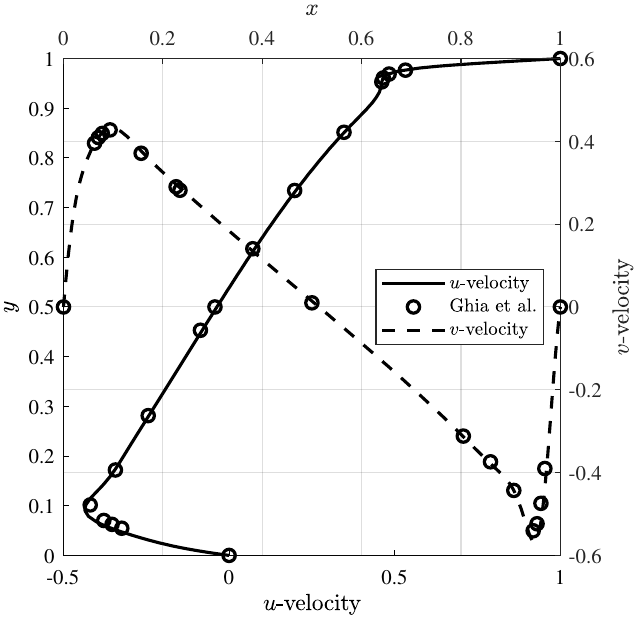}
        \caption{$\textrm{Re}=3,200$}\label{fig:y_u_x_v_Re_3200}
    \end{subfigure}\\
   \begin{subfigure}[t]{0.32\textwidth}
    \centering
        \includegraphics[trim={0.0cm 0.0cm 0.0cm 0.0cm}, clip=true,width=1.0\textwidth]{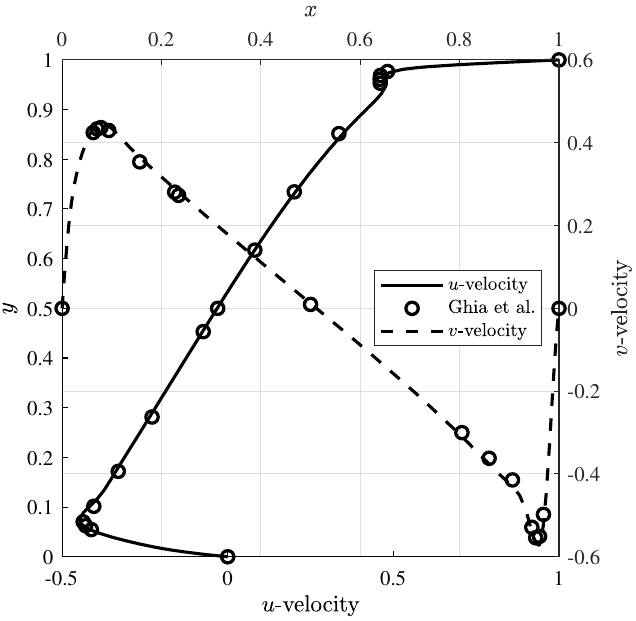}
        \caption{$\textrm{Re}=5,000$}\label{fig:y_u_x_v_Re_5000}
    \end{subfigure}%
    \centering
		\begin{subfigure}[t]{0.32\textwidth}
    \centering
        \includegraphics[trim={0.0cm 0.0cm 0.0cm 0.0cm}, clip=true,width=1.0\textwidth]{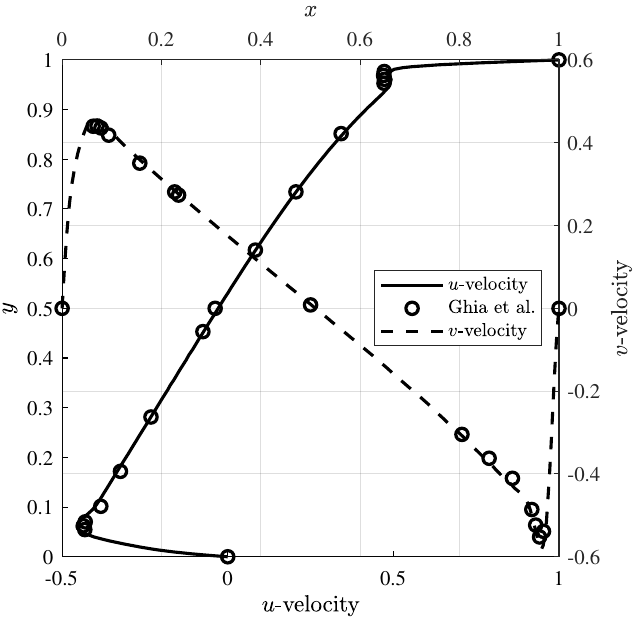}
        \caption{$\textrm{Re}=7,500$}\label{fig:y_u_x_v_Re_7500}
    \end{subfigure}%
  \begin{subfigure}[t]{0.32\textwidth}
    \centering
        \includegraphics[trim={0.0cm 0.0cm 0.0cm 0.0cm}, clip=true,width=1.0\textwidth]{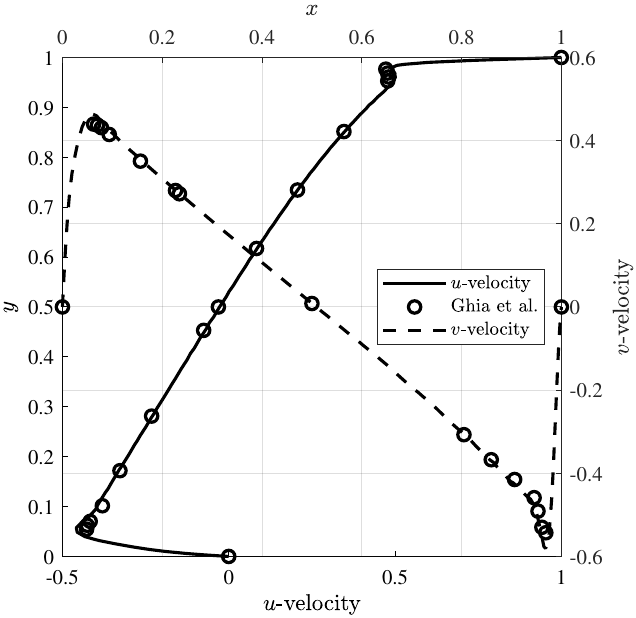}
        \caption{$\textrm{Re}=10,000$}\label{fig:y_u_x_v_Re_10000}
    \end{subfigure} 
 \caption{Comparisons of the numerically computed horizontal velocity $(u)$ and vertical velocity $(v)$ profile with the benchmark data (Ghia et al.~\cite{ghia1982high}) along $x=0.5$ and $y=0.5$ for different Reynolds numbers}
\label{Figs:LDC_y_u_x_v_comp_study}
\end{figure} \FloatBarrier
\begin{figure}
\centering
\begin{subfigure}{0.40\textwidth}
    \centering
        \includegraphics[trim={0.0cm 0.0cm 0.0cm 0.0cm}, clip=true,width=1.0\textwidth]{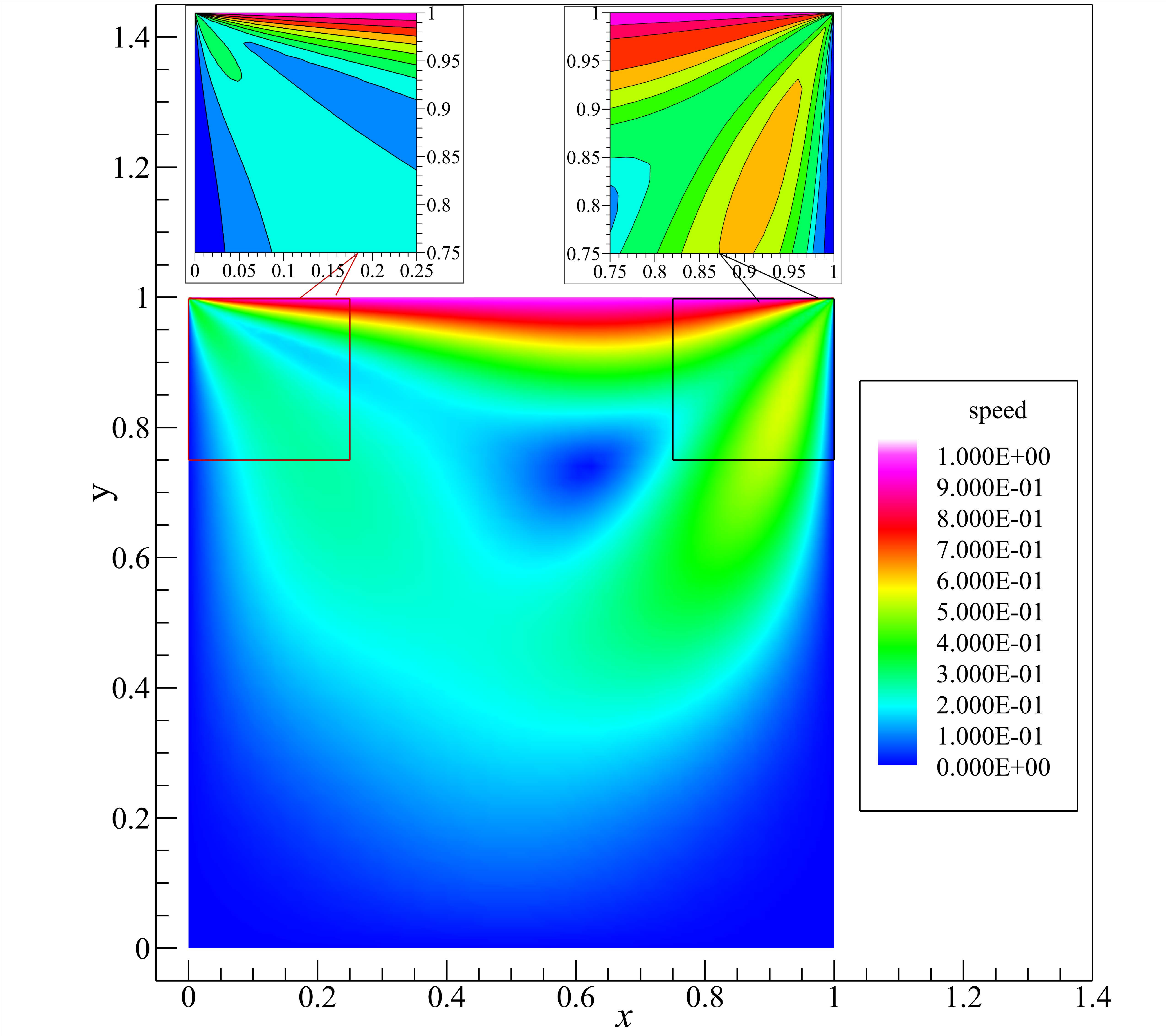}
        \caption{$\textrm{Re}=100$ }\label{Fig:Speed_Re100}
    \end{subfigure}%
		\begin{subfigure}{0.40\textwidth}
    \centering
        \includegraphics[trim={0.0cm 0.0cm 0.0cm 0.0cm}, clip=true,width=1.0\textwidth]{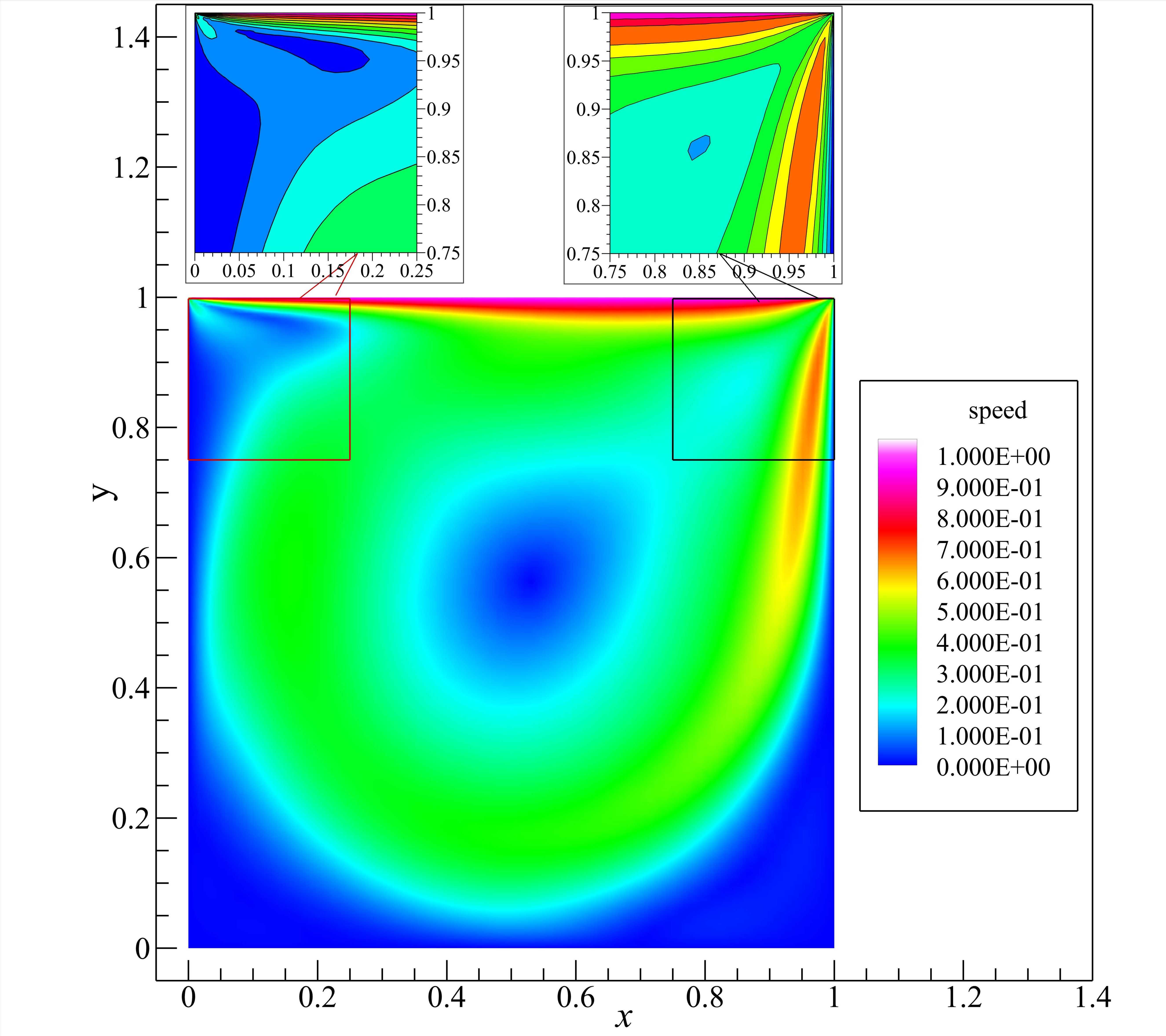}
        \caption{$\textrm{Re}=1,000$}\label{Fig:Speed_Re1000}
    \end{subfigure}\\
    \begin{subfigure}{0.40\textwidth}
    \centering
        \includegraphics[trim={0.0cm 0.0cm 0.0cm 0.0cm}, clip=true,width=1.0\textwidth]{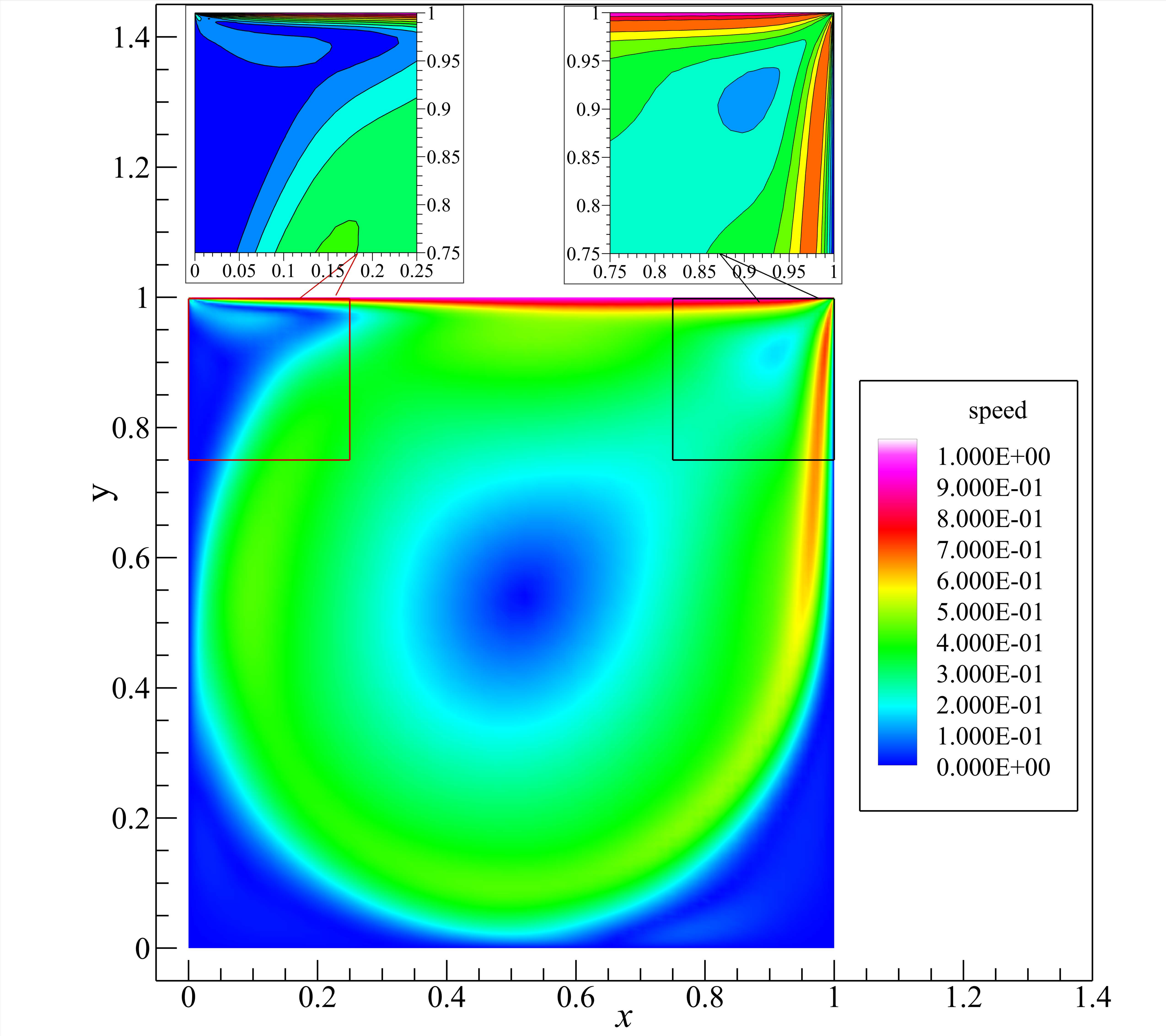}
        \caption{$\textrm{Re}=3,200$}\label{Fig:Speed_Re3200}
    \end{subfigure}%
        \begin{subfigure}{0.40\textwidth}
    \centering
        \includegraphics[trim={0.0cm 0.0cm 0.0cm 0.0cm}, clip=true,width=1.0\textwidth]{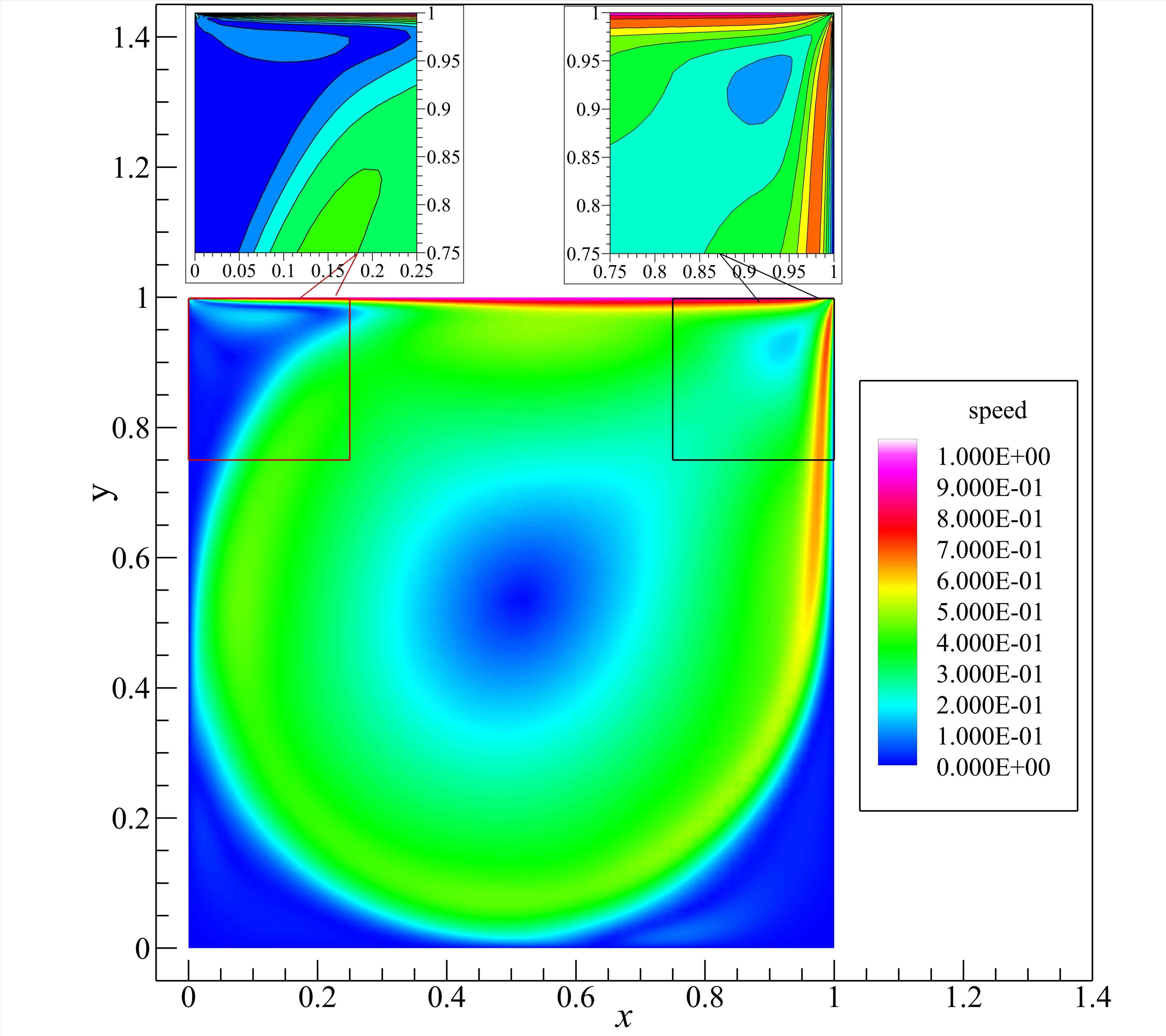}
        \caption{$\textrm{Re}=5,000$}\label{Fig:Speed_Re5000}
    \end{subfigure}\\
         \begin{subfigure}{0.40\textwidth}
    \centering
        \includegraphics[trim={0.0cm 0.0cm 0.0cm 0.0cm}, clip=true,width=1.0\textwidth]{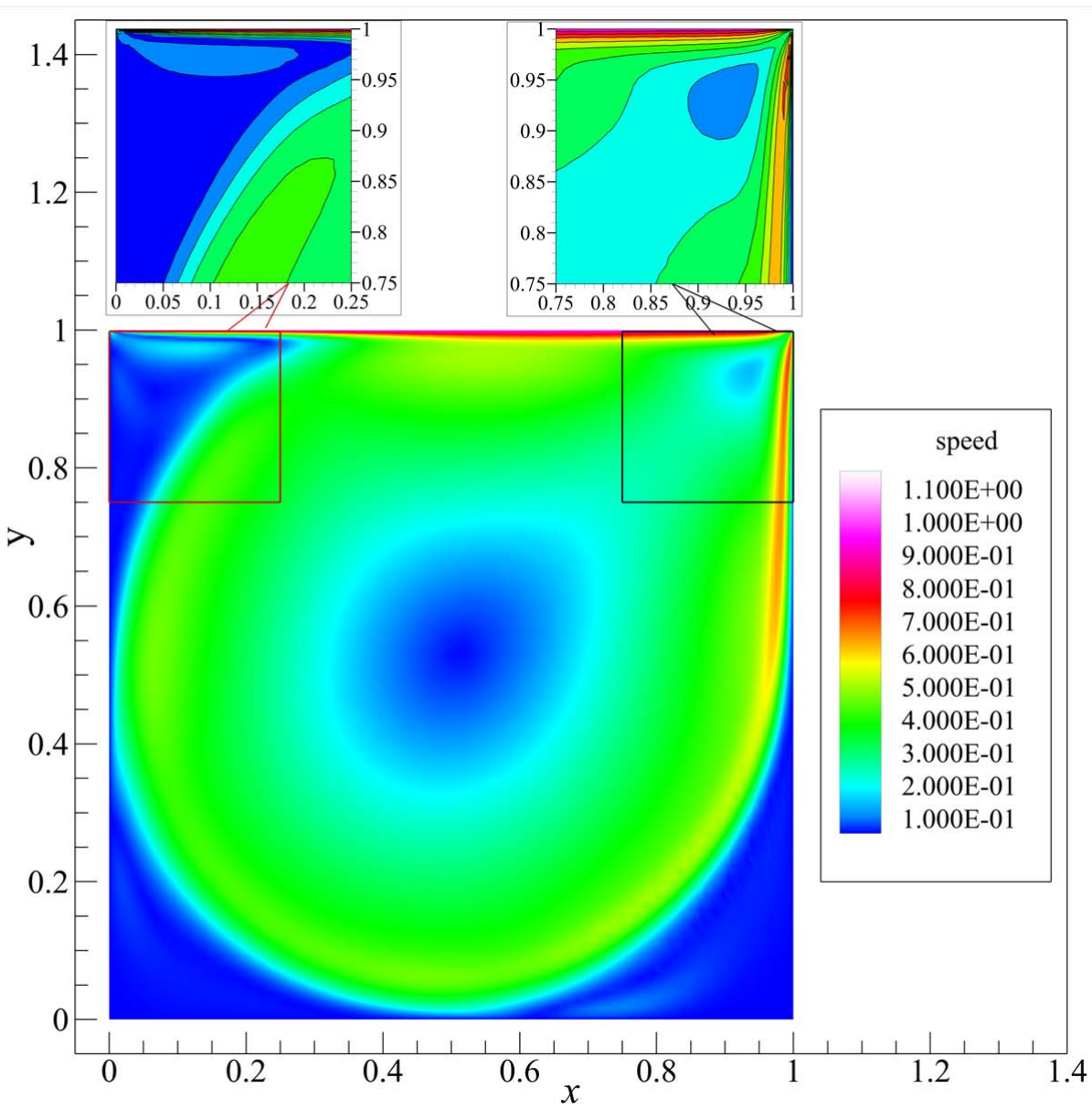}
        \caption{$\textrm{Re}=7,500$}\label{Fig:Speed_Re7500}
    \end{subfigure}%
             \begin{subfigure}{0.40\textwidth}
    \centering
        \includegraphics[trim={0.0cm 0.0cm 0.0cm 0.0cm}, clip=true,width=1.0\textwidth]{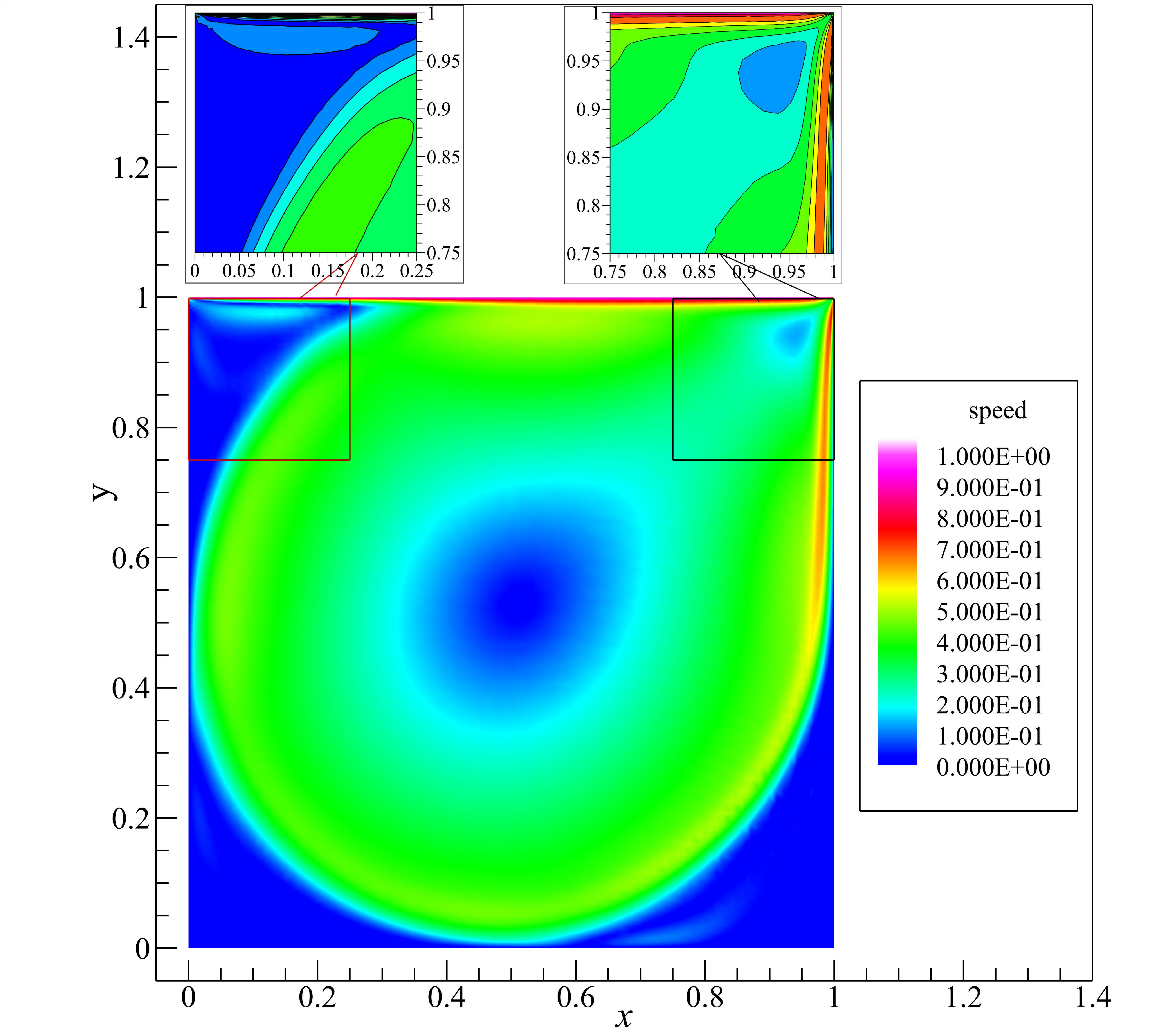}
        \caption{$\textrm{Re}=10,000$}\label{Fig:Speed_Re10000}
    \end{subfigure}
    \caption{Contour plots of speed profiles for different $\textrm{Re}$}
    \label{Figs:Speed_contour_diff_Re}
\end{figure}
\clearpage
\newpage
\noindent
\begin{figure}
\centering
\begin{subfigure}{0.27\textheight}
    \centering
        \includegraphics[trim={0.0cm 0.0cm 0.0cm 0.0cm}, clip=true,width=1.0\textwidth]{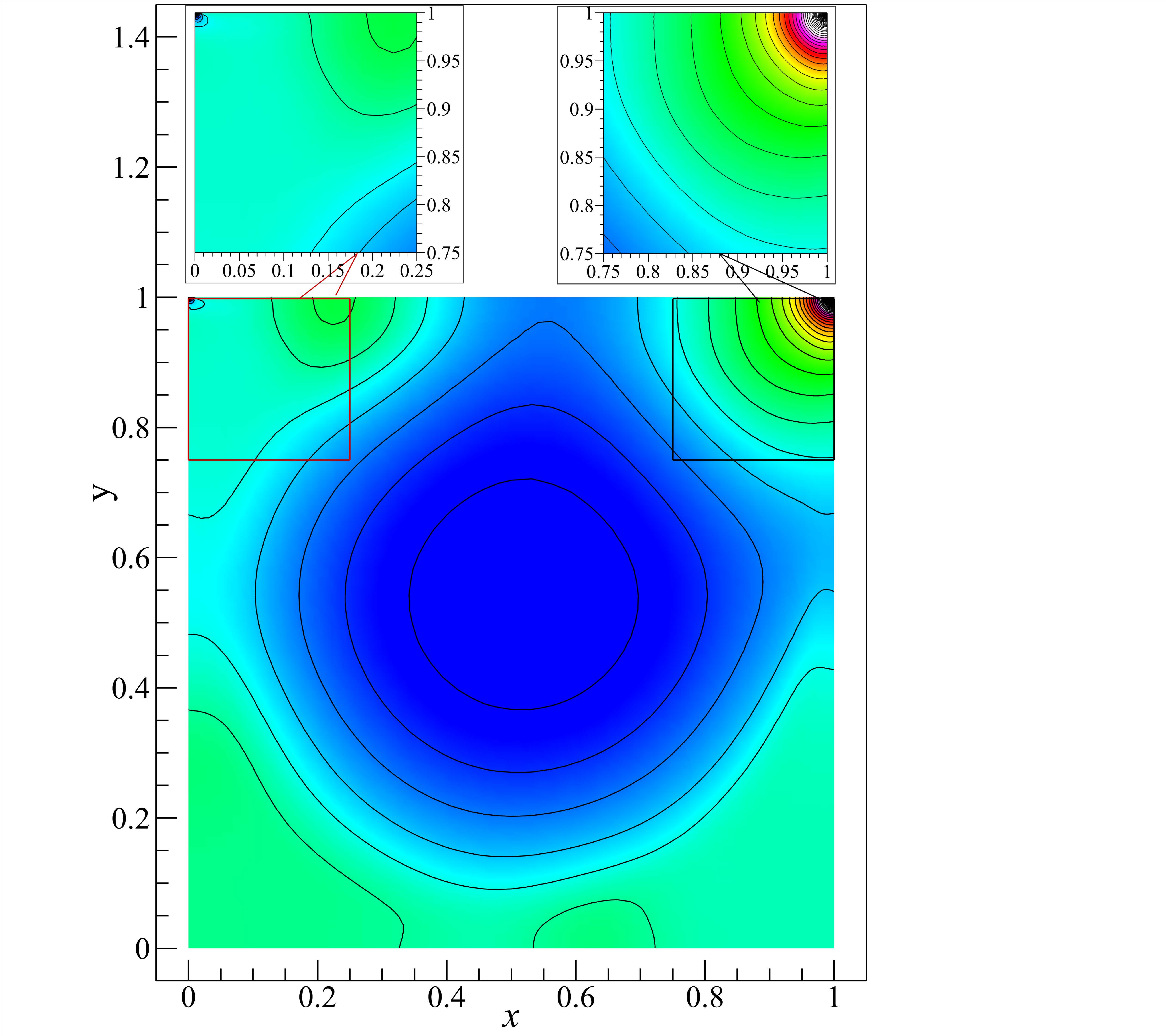} 
        \caption{$\textrm{Re}=3,200$ }\label{Fig:Pressure_Re3200}
\end{subfigure}%
		\begin{subfigure}{0.27\textheight}
    \centering
        \includegraphics[trim={0.0cm 0.0cm 0.0cm 0.0cm}, clip=true,width=1.0\textwidth]{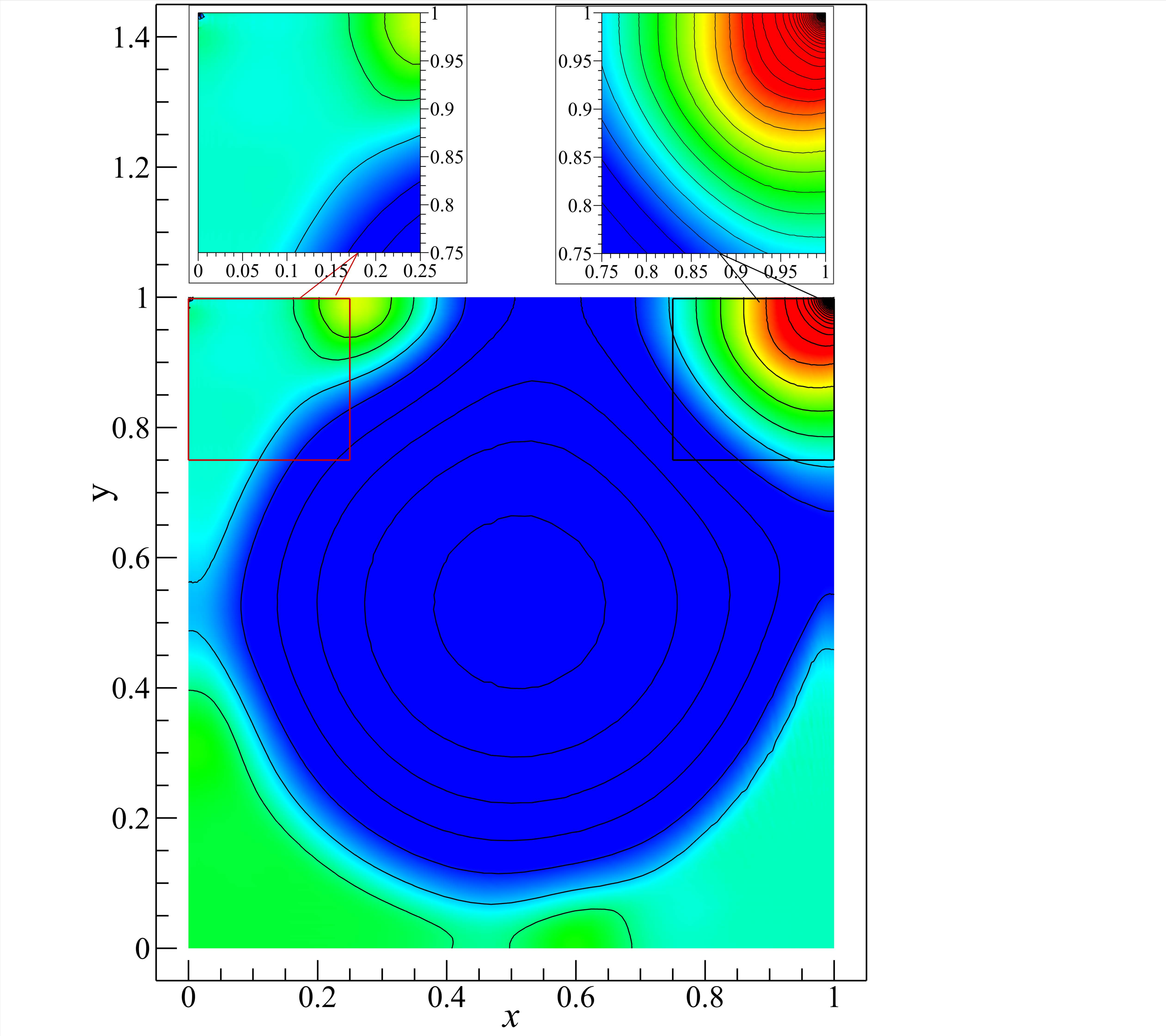}
        \caption{$\textrm{Re}=10,000$}\label{Fig:Pressure_Re1000}
    \end{subfigure}
    \caption{Pressure contours at $\textrm{Re}=100$ and 10,000}\label{Fig:Pressure_Re_3200_Re10000}
\end{figure}

Figure~\ref{Figs:Speed_contour_diff_Re} presents contour plots of the flow speed $= \sqrt{u^2 + v^2}$. The imposition of a unit tangential velocity along the top lid generates a high-speed region near the upper boundary, where the fluid is driven toward the right wall, redirected downward, and recirculates within the cavity. At $\textrm{Re} = 100$, a primary vortex is observed in the flow field. As the Reynolds number increases, this vortex gradually shifts toward the geometric center of the domain. Additionally, a secondary vortex forms near the top-right corner at $\textrm{Re} = 400$, and continues to intensify and migrate closer to the corner with increasing Reynolds number.


To assess the effectiveness of the SAT technique in handling the discontinuous boundary conditions in a stable manner, enlarged views of the speed profiles near the top corners are also provided. These demonstrate that the inclusion of the SAT term eliminates spurious oscillations near the top boundaries, ensuring smooth speed contours. Note that the speed profiles at $\textrm{Re}=7,500$ and $\textrm{Re}=10,000$ are almost identical. Furthermore, the pressure distributions at $\textrm{Re}=3,200$ and 10,000 at Figure~\ref{Fig:Pressure_Re_3200_Re10000} also exhibit smooth profiles without any indication of nonphysical oscillations~\cite{malan2002improved}, further supporting the robustness of the SAT technique. 
\subsection{Backward facing step flow problem}
\label{Sec:BFS_PROBLEM}
As a final example, we study the incompressible viscous fluid flow over a two-dimensional backward-facing step. This benchmark problem has been solved by numerous researchers to study different flow features, e.g. flow separation and reattachment of boundary layers. Here we consider the reference geometry, provided in \cite{gartling1990test,mandal2023weakly}. The schematic geometry and the associated boundary conditions are shown in Figure~\ref{Fig:Step_geom_BC}. The channel upstream portion of the step is excluded and replaced by an inflow boundary condition as per \cite{gartling1990test}. The channel height at the downstream region is equal to unit height $H$ and at the upstream side, it is set to $h=H/2$. The channel has length $L=30H$ and it extends to 60 step heights from the inlet. 
\begin{figure}
\centering
\includegraphics[trim={0.0cm 0.0cm 0.0cm 0.0cm}, clip=true,width=1.0\textwidth]{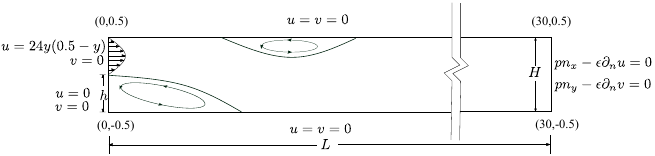}
        \caption{Schematic representation of backward-facing step geometry with boundary conditions}\label{Fig:Step_geom_BC}
\end{figure}

A solid wall boundary condition is prescribed at the top and bottom boundaries. A parabolic horizontal velocity profile $u(y)=24y(0.5-y)$ is imposed at the inflow boundary for $0\leq y\leq 0.5$, whereas for the remaining half of the west boundary we impose solid wall boundary conditions. The parabolic velocity profile produces a maximum inflow velocity $u_{\textrm{max}}=1.5$ and an average velocity of $u_{\textrm{avg}}=1.0$. For the outflow boundary at the east end, we use the natural boundary conditions derived and proven stable in \cite{nordstrom2019energy}.

\begin{figure}
\centering
\includegraphics[trim={0.0cm 0.0cm 0.0cm 0.0cm}, clip=true,width=1.0\textwidth]{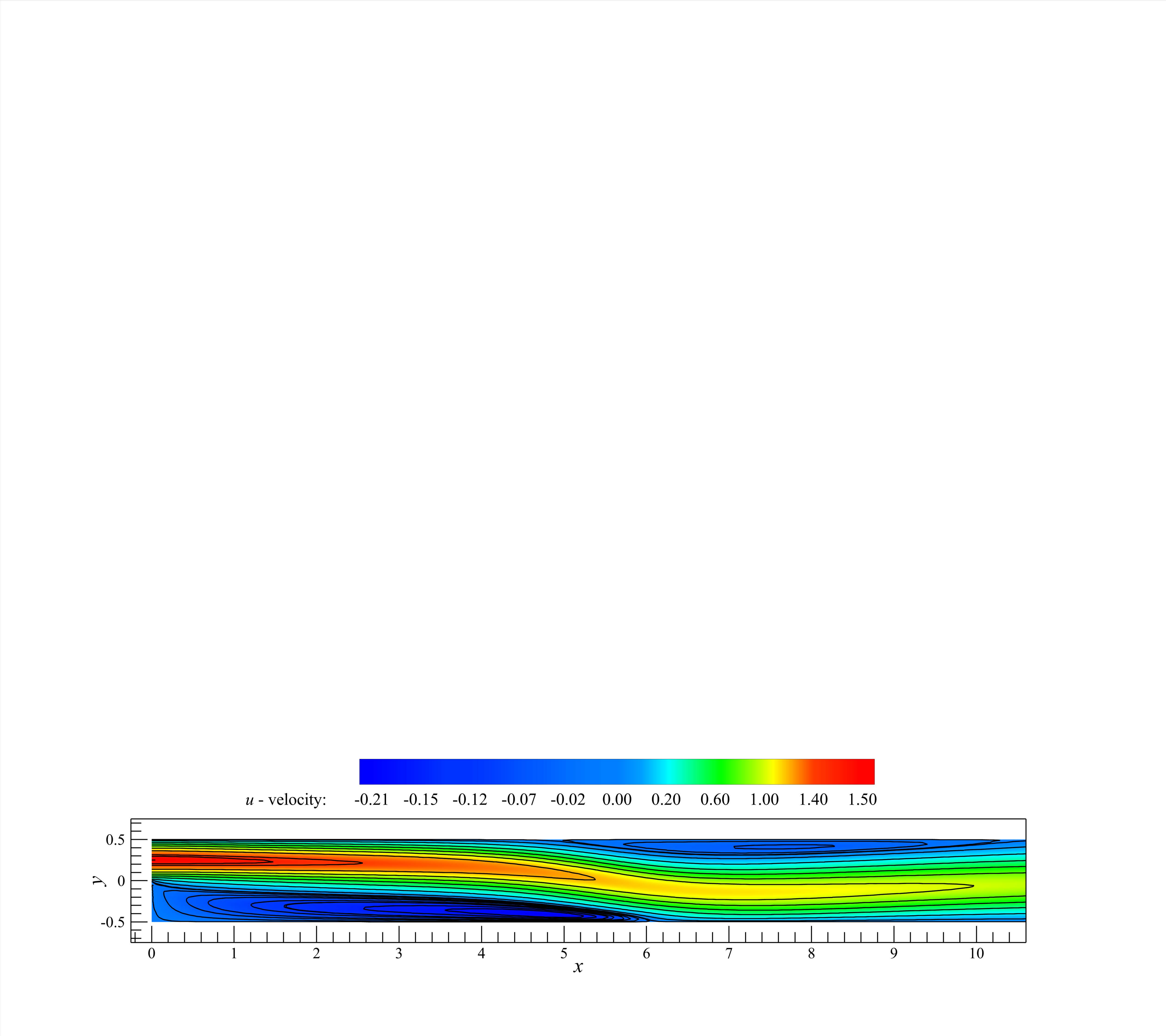} 
        \caption{Contour plots for $u$-velocity at $\textrm{Re}=800$}\label{Fig:Bfs_u_cntr_Re800}
\end{figure}
\begin{figure}
\centering
\includegraphics[trim={0.0cm 0.0cm 0.0cm 0.0cm}, clip=true,width=1.0\textwidth]{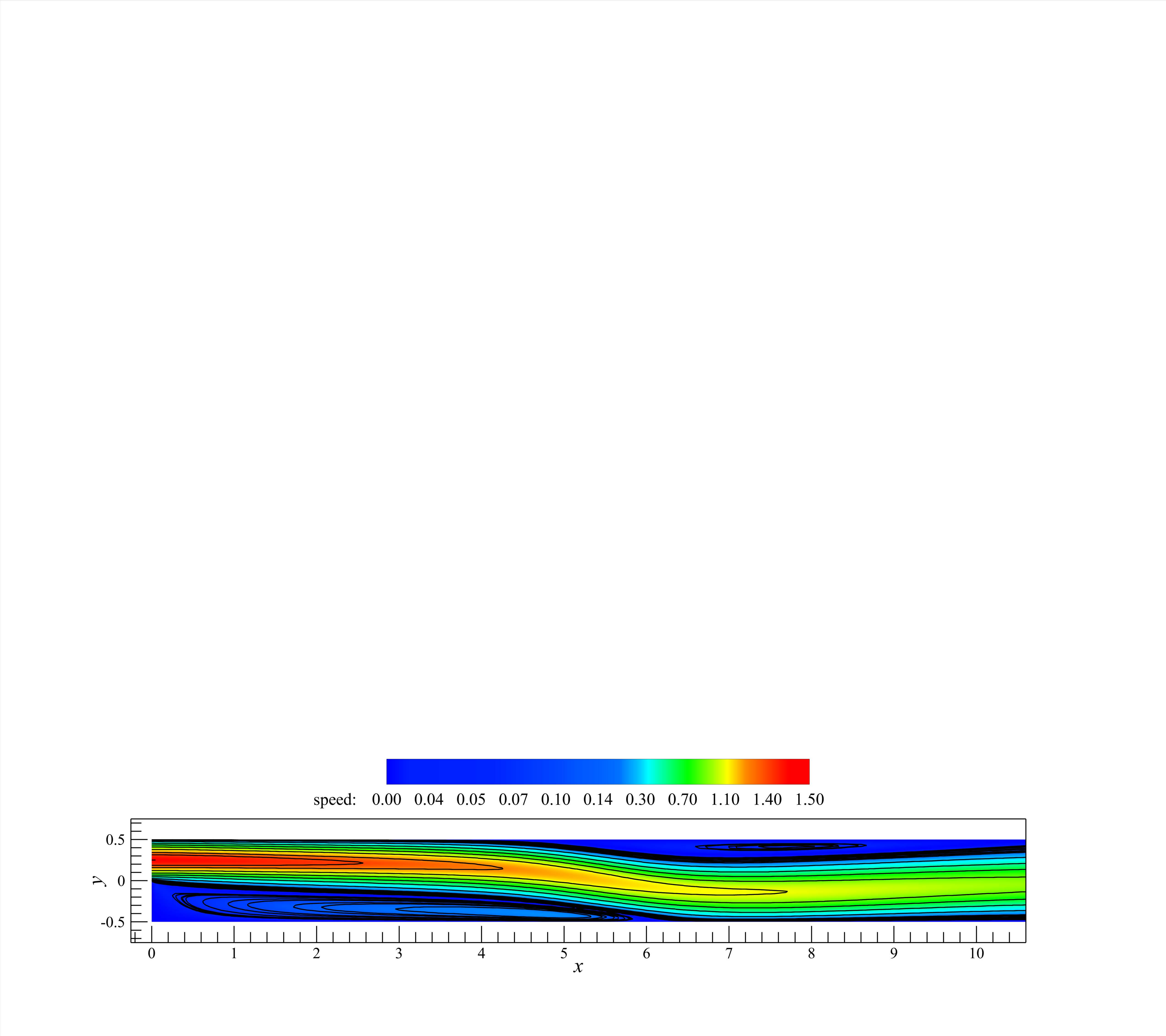}
        \caption{Contour plots for speed at $\textrm{Re}=800$}\label{Fig:Bfs_speed_cntr_Re800}
\end{figure}
\begin{figure}
\centering
\includegraphics[trim={0.0cm 0.0cm 0.0cm 0.0cm}, clip=true,width=1.0\textwidth]{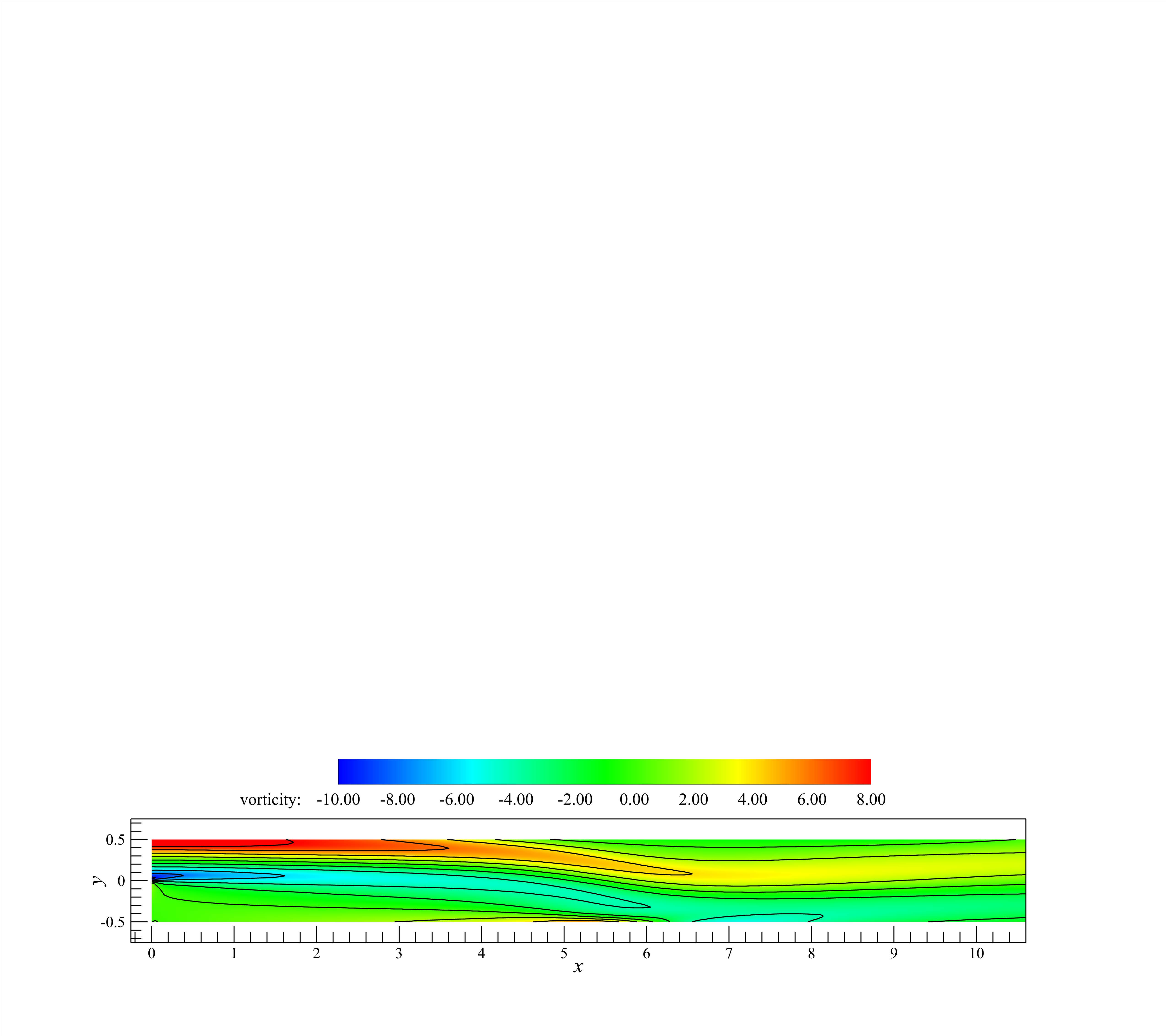}
        \caption{Vorticity contour at $\textrm{Re}=800$}\label{Fig:Bfs_vort_cntr_Re800}
\end{figure}
\begin{figure}
\centering
\includegraphics[trim={0.0cm 0.0cm 0.0cm 0.0cm}, clip=true,width=1.0\textwidth]{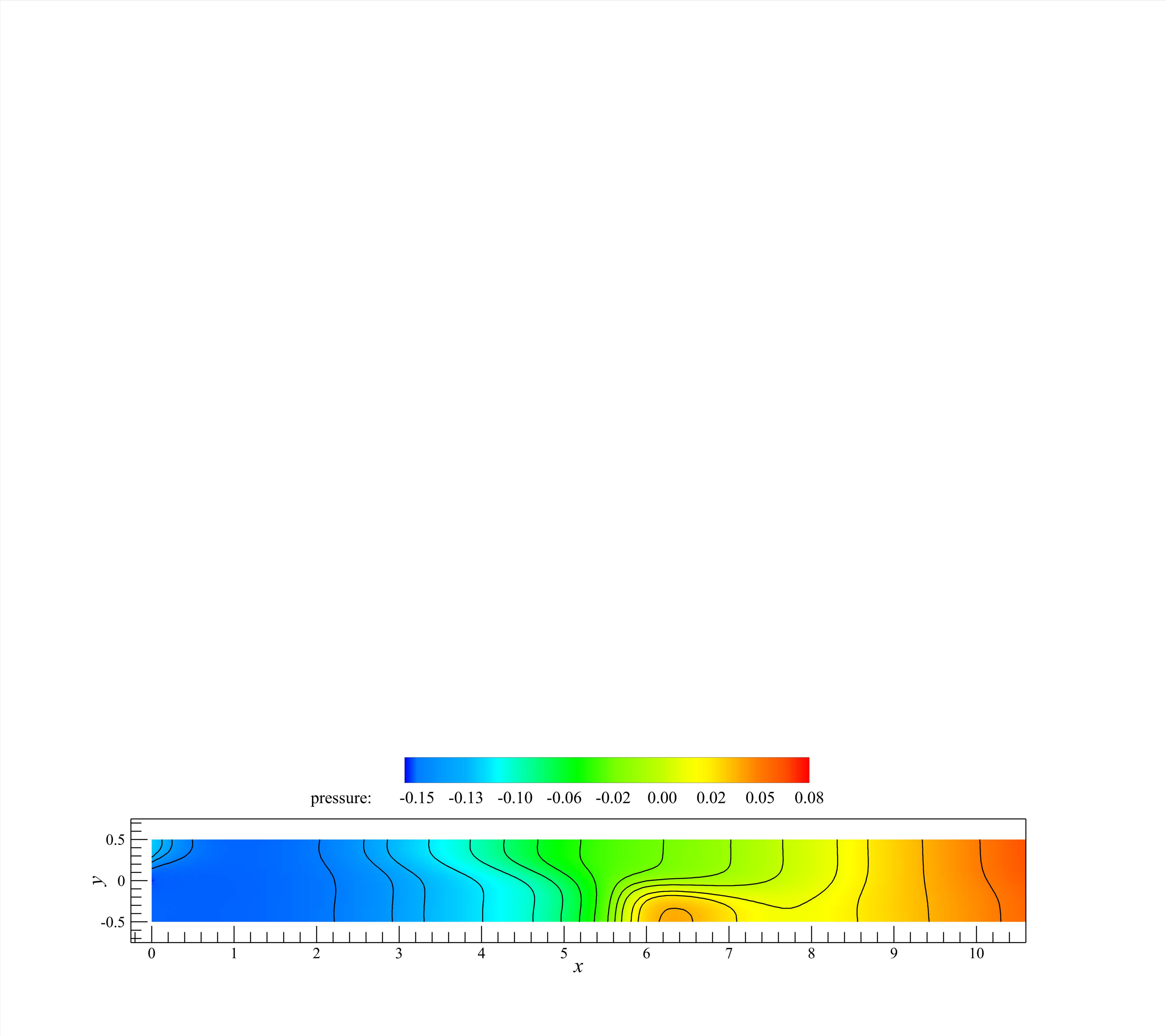}
        \caption{Pressure contour at $\textrm{Re}=800$}\label{Fig:Bfs_p_cntr_Re800}
\end{figure}

We discretize the two-dimensional domain using fourth-order Lagrange polynomials over a uniformly distributed mesh along both the channel length and width consisting of $100 \times 14$ elements, yielding a total of $401 \times 57=22857$ nodal points. We start with zero initial conditions and adopt the same iterative approach outlined previously with a time step size $\Delta t = 0.1$.

To assess the accuracy of our formulation, we compare the numerical results at $\textrm{Re} = 800$ with the reference results in \cite{gartling1990test}. Notably, the reference solution was obtained using a much finer mesh with second-order Lagrange polynomials for the velocity field consisting of 129681 nodes, i.e., almost six times more than we used. The contour plots shown in Figures~\ref{Fig:Bfs_u_cntr_Re800}-\ref{Fig:Bfs_p_cntr_Re800} describe the presence of the recirculation region. The primary and the secondary vortices are clearly visible from the speed and the $u$-velocity contour plots. The primary vortex is formed at the bottom wall of the channel near the step corner, while the secondary vortex is observed near the upper wall. These results are consistent with the results in \cite{gartling1990test,mandal2023weakly}, as are the vorticity $\left(\omega=\partial v/\partial x-\partial u/\partial y\right)$ and pressure profiles.
\begin{figure}
\centering
\begin{subfigure}{0.33\textwidth}
    \centering
        \includegraphics[trim={0.0cm 0.0cm 0.0cm 0.0cm}, clip=true,width=1.0\textwidth]{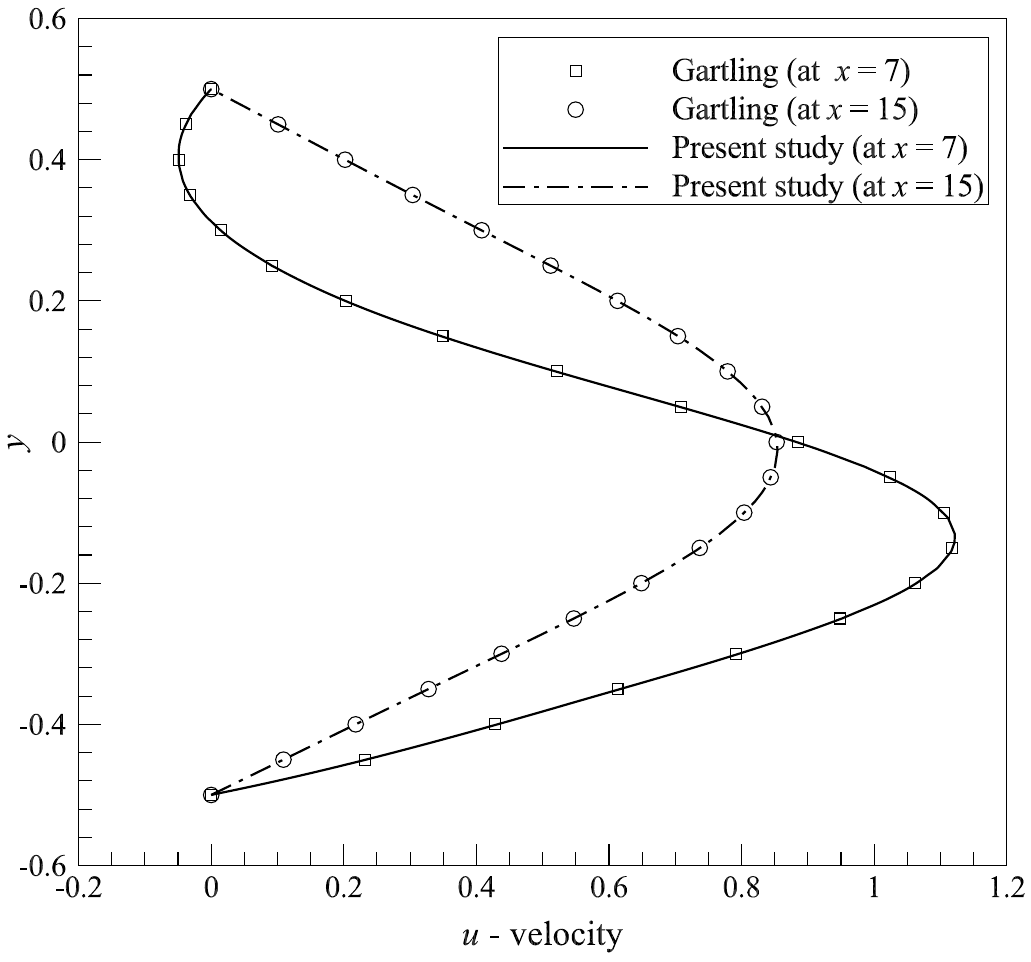}
        \caption{}\label{Fig:y_u_comp_gartling_x7x15_BFS}
    \end{subfigure}%
    \centering
		\begin{subfigure}{0.33\textwidth}
    \centering
        \includegraphics[trim={0.0cm 0.0cm 0.0cm 0.0cm}, clip=true,width=1.0\textwidth]{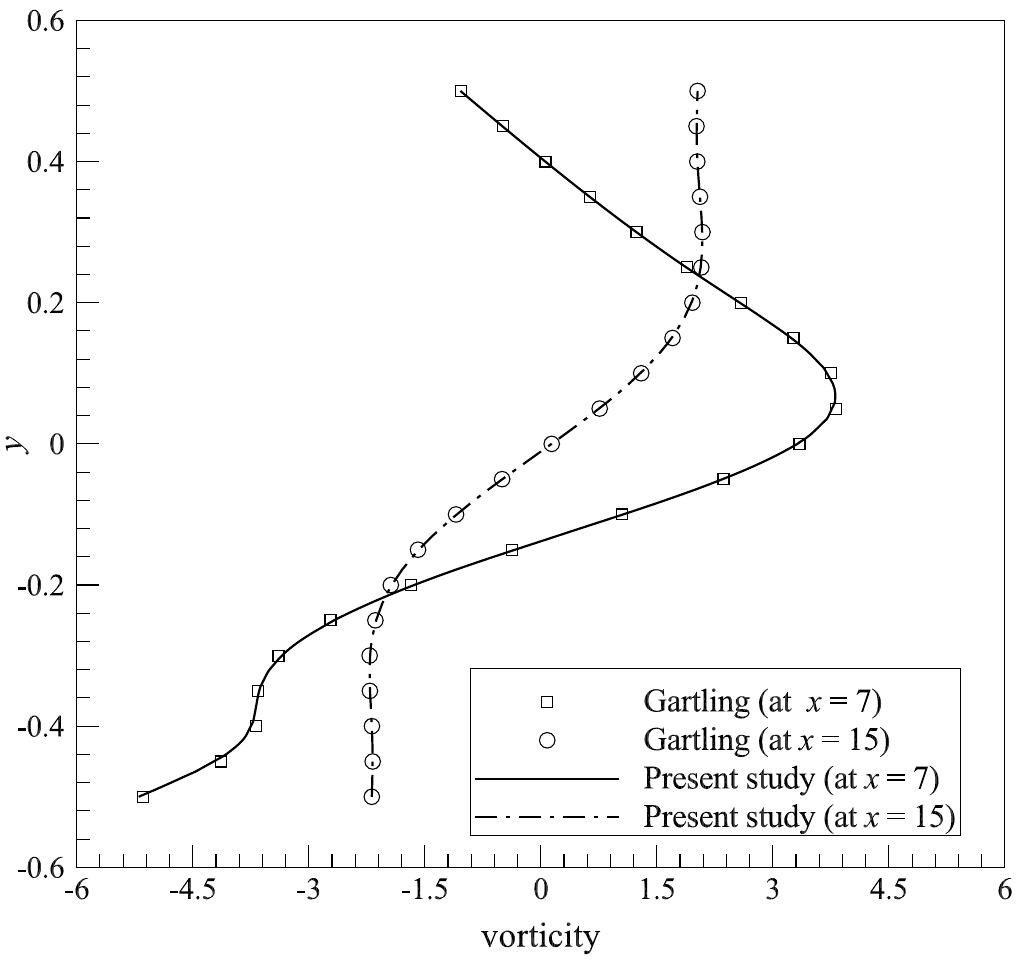}
        \caption{}\label{Fig:y_vort_comp_gartling_x7x15_BFS}
    \end{subfigure}
\centering
		\begin{subfigure}{0.33\textwidth}
    \centering
        \includegraphics[trim={0.0cm 0.0cm 0.0cm 0.0cm}, clip=true,width=1.0\textwidth]{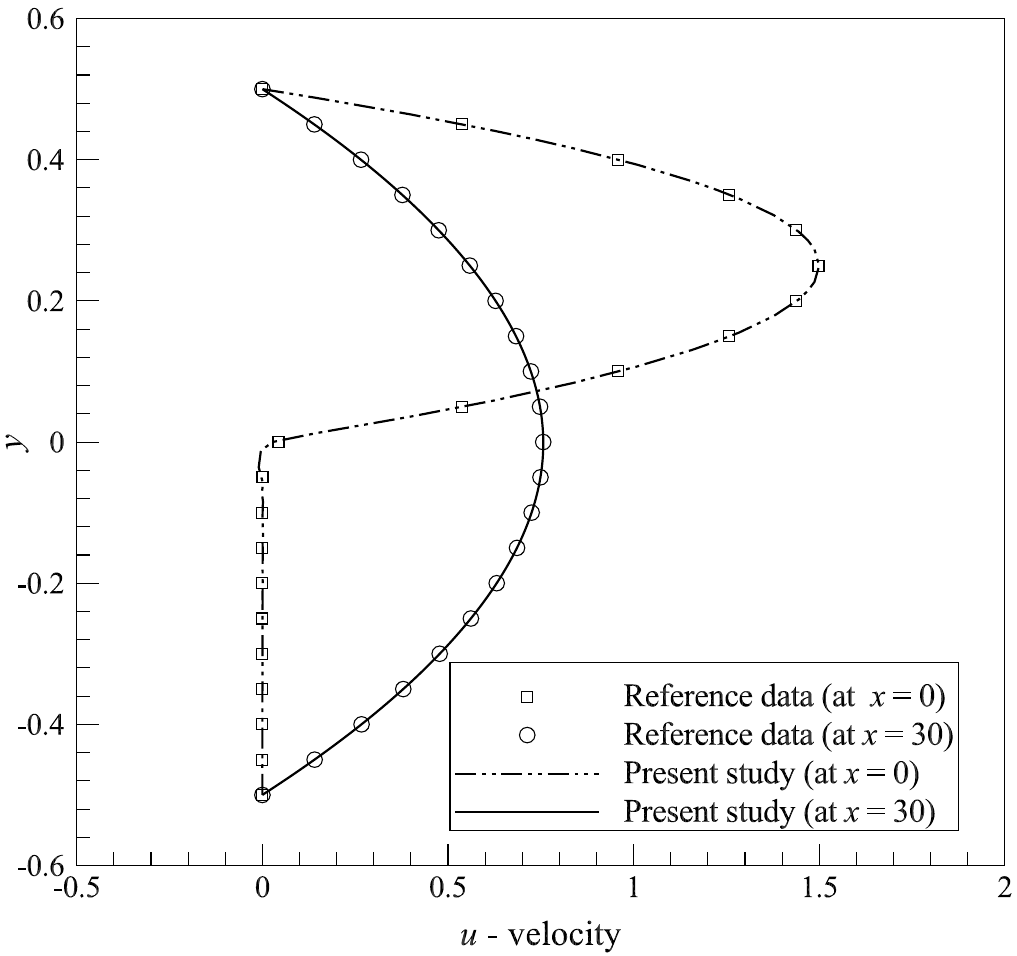} 
        \caption{}\label{Fig:y_u_inlet_outlet} 
    \end{subfigure}%
        \caption{Comparison of (a) $u$-velocity and (b) vorticity profiles at $x=7$ and $x=15$ with \cite{gartling1990test}, and (c) $u$-velocity variation along the inflow and outflow boundaries compared with the reference data \cite{mandal2023weakly} at $\textrm{Re}=800$.}
    \label{Figs:Comp_CrossFlwVariatins_u_vort_Gartling}
\end{figure}

The computed horizontal velocity and the vorticity in the cross flow directions compared to the benchmark data~\cite{gartling1990test} at $x=7$ and at $x=15$ are shown in Figures~\ref{Fig:y_u_comp_gartling_x7x15_BFS} and \ref{Fig:y_vort_comp_gartling_x7x15_BFS}, respectively. Excellent agreement is again achieved despite using the significantly coarser mesh. The variations of the $u$-velocity at the inflow and the outflow boundaries are plotted in Figure~\ref{Fig:y_u_inlet_outlet}. From the Figure~\ref{Fig:y_u_inlet_outlet}, it is clear that the horizontal velocity profile is gradually becoming parabolic, and it becomes fully developed at the outflow boundary where it matches exactly with the results in \cite{mandal2023weakly}.
\section{Summary and Conclusions}
\label{sec:Concluding_remarks}
A high-order CGFEM formulation in the SBP-SAT framework has been developed to solve initial-boundary value problems for the incompressible Navier–Stokes (INS) equations. The formulation is based on primitive variables, with equal-order Lagrange polynomials used for approximating both velocity and pressure fields. An energy analysis has been conducted at both the continuous and discrete levels, proving the stability of the method through the establishment of an energy bound.

Numerical validation has been performed using the MMS, confirming optimal order of convergence. The proposed method successfully handles discontinuous boundary conditions, as demonstrated by the lid-driven cavity flow problem, producing accurate numerical results without spurious oscillations across a wide range of Reynolds numbers from 100 to 10,000. 

To further assess the versatility of the formulation, it was also applied to the backward-facing step flow problem. The simulation accurately captured key flow features, including the formation of primary and secondary counter-rotating vortices at $\textrm{Re}=800$, thereby confirming the robustness and applicability of the proposed method to more general flow configurations. It also verified the efficiency of the new method.
\section*{Acknowledgements}
M.M and A.M. were financially supported by the National Research Foundation (NRF) of South Africa (Grant no. 89916)\footnote{Any opinion, findings and conclusions or recommendations expressed in this material are those of the author(s) and therefore the NRF do not accept any liability with regard thereto.} as well the HASTA project (Grant No. 101138003)\footnote{Views and opinions expressed are however those of the author(s) only and do not necessarily reflect those of the European Union. Neither the European Union nor the granting authority can be held responsible for them.} as part of the European Union Horizon research program. J.N. was supported by  Vetenskapsrådet, Sweden [award no. 2021-05484 VR] and University of Johannesburg Global Excellence
and Stature Initiative Funding.
\bibliographystyle{cas-model2-names}
\bibliography{CG_VS_FD_INS_BibFile}
\end{document}